\documentclass[reqno]{amsart}
\usepackage{amssymb}
\usepackage{mathrsfs} 
\usepackage{bm}
\usepackage{hyperref}
\usepackage{comment}
\usepackage{tikz}
\usetikzlibrary{arrows.meta}
\tikzset{
  .../.tip={[sep=2pt 1]
    Round Cap[]. Circle[length=0.5pt 1] Circle[length=0.5pt 1] Circle[length=0.5pt 1, sep=2pt]}}
    
\usepackage{pgffor}
\usetikzlibrary{shapes.misc}
\usepackage{pgflibraryshapes}
\usetikzlibrary{decorations.pathreplacing,calc,arrows.meta}
\tikzset{%
  half dotted/.style={
    decoration={show path construction, 
      lineto code={
          \draw[#1] (\tikzinputsegmentfirst) --($(\tikzinputsegmentfirst)!.3!(\tikzinputsegmentlast)$);,
          \draw[loosely dotted,-,#1] ($(\tikzinputsegmentfirst)!.5!(\tikzinputsegmentlast)$)--(\tikzinputsegmentlast);,
      }
    },
    decorate
  },
}

\renewcommand{\marginpar}[2][]{}

\newcommand{\Figure}{\textsc{Figure}\ }

\newcommand{\ZFC}{{\rm ZFC}}

\renewcommand{\emptyset}{\varnothing}

\renewcommand{\P}{{\mathbb P}}

\newcommand{\R}{{\mathscr R}}
\newcommand{\I}{{\mathscr I}}

\newcommand{\restrict}{\upharpoonright}
\newcommand{\concat}{\mathbin{{}^\smallfrown}}

\newcommand{\<}{\langle}
\renewcommand{\>}{\rangle}

\newcommand{\st}{\mid}

\newcommand{\ot}{\mathop{\rm ot}\nolimits}

\newcommand{\crit}{\mathop{\rm crit}}

\newcommand{\NS}{{\mathop{\rm NS}}}

\renewcommand{\and}{\mathop{\&}}

\newtheorem{theorem}{Theorem}[section]
\newtheorem{lemma}[theorem]{Lemma}
\newtheorem{corollary}[theorem]{Corollary}
\newtheorem{proposition}[theorem]{Proposition}

\theoremstyle{definition}

\newtheorem{remark}[theorem]{Remark}

\newtheorem{definition}[theorem]{Definition}

\newtheorem*{theorem_intro}{Theorem 1.2}
\thanks{The author would like to thank Sean Cox, Victoria Gitman and Chris Lambie-Hanson for many detailed conversations related to the topics of this work. Specifically, the author thanks Victoria Gitman for suggesting the proofs of Theorem \ref{theorem_ramsey_equiv} and Theorem \ref{theorem_ramsey_reflection}. Additionally, the author thanks Joan Bagaria, Philip Welch and the anonymous referee for their helpful comments.}
\date{\today}

\begin{document}

\title{A refinement of the Ramsey hierarchy via indescribability}

\author[Brent Cody]{Brent Cody}
\address[Brent Cody]{ 
Virginia Commonwealth University,
Department of Mathematics and Applied Mathematics,
1015 Floyd Avenue, PO Box 842014, Richmond, Virginia 23284, United States
} 
\email[B. ~Cody]{bmcody@vcu.edu} 
\urladdr{http://www.people.vcu.edu/~bmcody/}

\begin{abstract}

A subset $S$ of a cardinal $\kappa$ is Ramsey if for every function $f:[S]^{<\omega}\to \kappa$ with $f(a)<\min a$ for all $a\in[S]^{<\omega}$, there is a set $H\subseteq S$ of cardinality $\kappa$ which is \emph{homogeneous} for $f$, meaning that $f\restrict[H]^n$ is constant for each $n<\omega$. Baumgartner proved \cite{MR0384553} that if $\kappa$ is a Ramsey cardinal, then the collection of non-Ramsey subsets of $\kappa$ is a normal ideal on $\kappa$. Sharpe and Welch \cite{MR2817562}, and independently Bagaria \cite{MR3894041}, extended the notion of $\Pi^1_n$-indescribability where $n<\omega$ to that of $\Pi^1_\xi$-indescribability where $\xi\geq\omega$. We study large cardinal properties and ideals which result from Ramseyness properties in which homogeneous sets are demanded to be $\Pi^1_\xi$-indescribable. By iterating Feng's Ramsey operator \cite{MR1077260} on the various $\Pi^1_\xi$-indescribability ideals, we obtain new large cardinal hierarchies and corresponding nonlinear increasing hierarchies of normal ideals. We provide a complete account of the containment relationships between the resulting ideals and show that the corresponding large cardinal properties yield a strict linear refinement of Feng's original Ramsey hierarchy. We also show that, given any ordinals $\beta_0,\beta_1<\kappa$ the increasing chains of ideals obtained by iterating the Ramsey operator on the $\Pi^1_{\beta_0}$-indescribability ideal and the $\Pi^1_{\beta_1}$-indescribability ideal respectively, are eventually equal; moreover, we identify the least degree of Ramseyness at which this equality occurs. As an application of our results we show that one can characterize our new large cardinal notions and the corresponding ideals in terms of generic elementary embeddings; as a special case this yields generic embedding characterizations of $\Pi^1_\xi$-indescribability and Ramseyness.

\end{abstract}



\maketitle




\section{Introduction}\label{section_introduction}

In his work on decidability problems, Ramsey \cite{MR1576401} proved his famous combinatorial theorem which states that if $m,n<\omega$ and $f:[\omega]^m\to n$ is a function then $f$ has an infinite \emph{homogeneous set} $H\subseteq\omega$, meaning that $f\restrict[H]^m$ is constant. The investigation of analogues of Ramsey's theorem for uncountable sets begun by Erd\H{o}s, Hajnal, Tarski, Rado and others (see \cite{MR0008249}, \cite{MR0065615}, \cite{MR81864} and \cite{MR95124}), quickly led to the definition of many large cardinal notions including weak compactness, Ramseyness, measurability and strong compactness (see \cite[Section 7]{MR1994835} for an account of the emergence of certain large cardinal axioms from the theory of partition relations). We say that $\kappa>\omega$ is a \emph{Ramsey cardinal} if for every function $f:[\kappa]^{<\omega}\to 2$ there is a set $H\subseteq\kappa$ of size $\kappa$ which is homogeneous for $f$, meaning that $f\restrict[H]^n$ is constant for all $n<\omega$.\footnote{See \cite{MR2830415} for additional motivation and an explanation of how Ramsey cardinals fit into the large cardinal hierarchy.} The study of Ramsey-like properties of uncountable cardinals has been a central concern of set theorists working on large cardinals and infinitary combinatorics, with renewed interest in recent years (see  \cite{MR0540770}, \cite{MR534574}, \cite{MR1077260}, \cite{MR2817562}, \cite{MR2830415}, \cite{MR2830435}, \cite{MR3348040}, \cite{MR3800756}, \cite{MR3922802}, \cite{carmody_gitman_habic} and \cite{Holy-Lucke}). In this article, we study Ramsey-like properties of uncountable cardinals in which homogeneous sets are demanded to have degrees of indescribability: for example, a cardinal $\kappa$ is \emph{$1$-$\Pi^1_n$-Ramsey} where $n<\omega$ if and only if every function $f:[\kappa]^{<\omega}\to 2$ has a $\Pi^1_n$-indescribable homogeneous set $H\subseteq\kappa$. Among other things, we show that hypotheses of this kind and their generalizations lead to a strict refinement of Feng's \cite{MR1077260} original Ramsey hierarchy: we isolate large cardinal hypotheses which provide strictly increasing hierarchies between Feng's $\Pi_\alpha$-Ramsey and $\Pi_{\alpha+1}$-Ramsey cardinals for all $\alpha<\kappa$.

Baumgartner showed (see \cite{MR0384553} and \cite{MR0540770}) that in many cases large cardinal properties can be viewed as properties of subsets of cardinals and not just of the cardinals themselves. Recall that for $S\subseteq\kappa$ where $\kappa$ is a cardinal, a function $f:[S]^{<\omega}\to\kappa$ is \emph{regressive} if $f(a)<\min a$ for all $a\in[S]^{<\omega}$. It is well-known (see \cite[Section 4]{MR0540770} or \cite[Lemma 2.42]{MR2710923}) that $\kappa$ is a Ramsey cardinal if and only if for every regressive function $f:[\kappa]^{<\omega}\to \kappa$ there is a set $H\subseteq\kappa$ of size $\kappa$ which is homogeneous for $f$. A set $S\subseteq\kappa$ is \emph{Ramsey} if every regressive function $f:[S]^{<\omega}\to \kappa$ has a homogeneous set $H\subseteq S$ of size $\kappa$.\footnote{Let us point out here that several authors, including Baumgartner \cite{MR0540770} and Feng \cite{MR1077260}, use a different definition of Ramsey set which is equivalent to ours (see Proposition \ref{proposition_4_5_6} below): in \cite{MR0540770}, a set $S\subseteq\kappa$ is Ramsey if for every club $C\subseteq\kappa$ and every regressive function $f:[S]^{<\omega}\to\kappa$ there is a set $H\subseteq S\cap C$ of size $\kappa$ which is homogeneous for $f$.} This leads naturally to the consideration of large cardinal ideals: for example, Baumgartner showed that if $\kappa$ is a Ramsey cardinal then the collection of non-Ramsey subsets of $\kappa$ is a nontrivial normal ideal on $\kappa$ called the \emph{Ramsey ideal}. Similarly, a set $S\subseteq\kappa$ is \emph{$\Pi^1_n$-indescribable} if for all $A\subseteq V_\kappa$ and all $\Pi^1_n$ sentences $\varphi$, if $(V_\kappa,\in,A)\models\varphi$ then there is an $\alpha\in S$ such that $(V_\alpha,\in,A\cap V_\alpha)\models\varphi$, and the collection $\Pi^1_n(\kappa)$ of non--$\Pi^1_n$-indescibable subsets of $\kappa$ is a normal ideal on $\kappa$ when $\kappa$ is $\Pi^1_n$-indescribable. Baumgartner proved that a cardinal $\kappa$ is Ramsey if and only if $\kappa$ is \emph{pre-Ramsey},\footnote{Pre-Ramseyness is defined below in Section \ref{section_feng}.} $\kappa$ is $\Pi^1_1$-indescribable and the union of the $\Pi^1_1$-indescribability ideal on $\kappa$ and the pre-Ramsey ideal on $\kappa$ generate a nontrivial ideal\footnote{An ideal on $\kappa$ is nontrivial if it is not equal to the entire powerset of $\kappa$.} which equals the Ramsey ideal; furthermore, reference to these ideals cannot be removed from this characterization because the least cardinal which is both pre-Ramsey and $\Pi^1_1$-indescribable is not Ramsey. Thus, Baumgartner's work shows that consideration of large cardinal ideals is, in a sense, necessary for certain results.

Generalizing the definition of Ramseyness, Feng \cite{MR1077260} defined the \emph{Ramsey operator} $\R$ as follows. Given an ideal $I\supseteq[\kappa]^{<\kappa}$ on $\kappa$, we define an ideal $\R(I)$ on $\kappa$ by letting $S\notin\R(I)$ if and only if for every regressive function $f:[S]^{<\omega}\to \kappa$ there is a set $H\in I^+$ homogeneous for $f$. It is easy to see that $I\subseteq\R(I)$ and that $I\subseteq J$ implies $\R(I)\subseteq \R(J)$ for all ideals $I,J\supseteq[\kappa]^{<\kappa}$ on $\kappa$. Feng proved that if $\kappa$ is a regular cardinal and $I\supseteq[\kappa]^{<\kappa}$ is an ideal on $\kappa$, then $\R(I)$ is a normal ideal on $\kappa$.\footnote{Feng used a different definition of $\R(I)$ which is equivalent to ours when either $I\supseteq\NS_\kappa$ or $I=[\kappa]^{<\kappa}$ (see Theorem \ref{theorem_ramsey_equiv} below).} Notice that $\kappa$ is a Ramsey cardinal if and only if $\kappa\notin\R([\kappa]^{<\kappa})$, and in this case $\R([\kappa]^{<\kappa})$ is the Ramsey ideal on $\kappa$. Building on Baumgartner's work \cite{MR0540770} on the ineffability hierarchy below a completely ineffable cardinal, Feng showed that one can iterate the Ramsey operator to obtain an increasing chain of ideals on $\kappa$ corresponding to a strict hierarchy of large cardinals as follows. Define $I^\kappa_{-2}=[\kappa]^{<\kappa}$ and $I^\kappa_{-1}=\NS_\kappa$. For $n<\omega$ let $I^\kappa_n=\R(I^\kappa_{n-2})$. Let $I^\kappa_{\alpha+1}=\R(I^\kappa_\alpha)$. If $\alpha$ is a limit ordinal let $I^\kappa_\alpha=\bigcup_{\xi<\alpha}I^\kappa_\xi$. It may at first appear strange that Feng's definition of $I^\kappa_n$ refers to $\NS_\kappa$ for odd $n<\omega$. We will return to this issue below in Remark \ref{remark_fengs_defintion} after introducing some notation which clarifies this issue and which will be important for the rest of the paper. In Feng's terminology,\footnote{We tend to avoid Feng's terminology because his ``$\Pi_\alpha$-Ramsey'' notation may create confusion with notation we employ for Ramsey properties defined using the $\Pi^1_\xi$-indescribability ideals.} a cardinal $\kappa$ is \emph{$\Pi_\alpha$-Ramsey} if and only if $\kappa\notin I^\kappa_\alpha$ and $\kappa$ is \emph{completely Ramsey} if and only if $\kappa\notin I^\kappa_\alpha$ for all $\alpha$. Generalizing a result of Baumgartner, Feng proved that $I^\kappa_m\supseteq\Pi^1_{m+1}(\kappa)$ for $1\leq m <\omega$, and as a consequence the axioms ``$\exists\kappa$($\kappa$ is $\Pi_n$-Ramsey)'' form a strictly increasing hierarchy. Using canonical functions, which were introduced by Baumgartner \cite{MR0540770} in his study of the ineffability hierarchy, Feng proved that this hierarchy of large cardinal axioms can be extended to obtain a strictly increasing hierarchy of axioms of the form ``$\exists\kappa$($\kappa$ is $\Pi_\alpha$-Ramsey)''. Moreover, Feng gave characterizations of the $\Pi_n$-Ramsey cardinals for $n<\omega$ in terms of indescribability ideals, which are similar to Baumgartner's above mentioned characterization of Ramseyness in that they use generalizations of pre-Ramseyness and the reference to ideals in the characterizations cannot be removed.


We introduce some notation that differs slightly from Feng's and which simplifies the presentation of our results. For an ideal $I\supseteq[\kappa]^{<\kappa}$ we define $\R^\alpha(I)$ for all ordinals $\alpha$ as follows. Let $\R^0(I)=I$. Assuming $\R^\alpha(I)$ has been defined let $\R^{\alpha+1}(I)=\R(\R^\alpha(I))$. If $\alpha$ is a limit ordinal, let $\R^\alpha(I)=\bigcup_{\xi<\alpha}\R^\xi(I)$. Feng's increasing chain of ideals can then be written as
\[[\kappa]^{<\kappa}\subseteq\NS_\kappa\subseteq\R([\kappa]^{<\kappa})\subseteq\R(\NS_\kappa)\subseteq\R^2([\kappa]^{<\kappa})\subseteq\R^2(\NS_\kappa)\subseteq\cdots.\tag{F}\label{feng}\]
\begin{remark}\label{remark_fengs_defintion}
Notice that $\R^\omega([\kappa]^{<\kappa})=\R^\omega(\NS_\kappa)$, and thus $\R^\alpha([\kappa]^{<\kappa})=\R^\alpha(\NS_\kappa)$ for $\alpha\geq\omega$. 
\end{remark}


Sharpe and Welch \cite[Definition 3.21]{MR2817562} extended the notion of $\Pi^1_n$-indescribability of a cardinal $\kappa$ where $n<\omega$ to that of $\Pi^1_\xi$-indescribability where $\xi<\kappa^+$ by demanding that the existence of a winning strategy for a particular player in a certain finite game played at $\kappa$ implies that the same player has a winning strategy in the analogous game played at some cardinal less than $\kappa$. Later, Bagaria \cite[Definition 4.2]{MR3894041} gave an alternative definition of the $\Pi^1_\xi$-indescribability of a cardinal $\kappa$ for $\xi<\kappa$ using the indescribability of rank-initial segments of the set-theoretic universe by certain sentences in an infinitary logic. In what follows we will use Bagaria's definition since it seems easier to work with in this context. Bagaria extended the definitions of the classes of $\Pi^1_n$ and $\Sigma^1_n$ formulas to define the natural classes of $\Pi^1_\xi$ and $\Sigma^1_\xi$ formulas for all ordinals $\xi$. For example, a formula is $\Pi^1_\omega$ if it is of the form $\bigwedge_{n<\omega}\varphi_n$ where each $\varphi_n$ is $\Pi^1_n$ and it contains finitely-many free second order variables.\footnote{See Section \ref{section_bagaria} below or \cite{MR3894041} for details.} A set $S\subseteq\kappa$ is said to be \emph{$\Pi^1_\xi$-indescribable} if for all $A\subseteq V_\kappa$ and all $\Pi^1_\xi$ sentences $\varphi$, if $(V_\kappa,\in,A)\models\varphi$ then there is some $\alpha\in S$ such that $(V_\alpha,\in,A\cap V_\alpha)\models\varphi$. Furthermore, Bagaria showed that when $\xi>0$, if $\kappa$ is $\Pi^1_\xi$-indescribable then the collection
\[\Pi^1_\xi(\kappa)=\{X\subseteq\kappa\st\text{$X$ is not $\Pi^1_\xi$-indescribable}\}\]
is a nontrivial normal ideal on $\kappa$. As a matter of notational convenience we let $\Pi^1_{-1}(\kappa)=[\kappa]^{<\kappa}$.


In this article we study ideals of the form $\R^\alpha(\Pi^1_\beta(\kappa))$ for ordinals $\alpha,\beta<\kappa$ and the corresponding hierarchy of large cardinals, which provides a strict refinement of Feng's original hierarchy.\footnote{We restrict the values of $\alpha$ and $\beta$ for which we consider $\R^\alpha(\Pi^1_\beta(\kappa))$ to be less than $\kappa$ because, as explained in Section \ref{section_bagaria}, for Bagaria's version of indescribability, if the $\Pi^1_\beta$-indescribability ideal $\Pi^1_\beta(\kappa)$ is nontrivial then $\beta<\kappa$, and if $\alpha\geq\kappa$ and $\beta<\kappa$ then $\R^\alpha(\Pi^1_{\beta}(\kappa))=\R^\alpha([\kappa]^{<\kappa})$ by Theorem \ref{theorem_culmination} and Corollary \ref{corollary_main_redundnacy}. Thus, apparently, consideration of the ideals $\R^\alpha(\Pi^1_\beta(\kappa))$ for $\alpha\geq\kappa$ and $\beta<\kappa$ is redundant given Feng's work on $\R^\alpha([\kappa]^{<\kappa})$.} For $\alpha,\beta<\kappa$ we say that $\kappa$ is \emph{$\alpha$-$\Pi^1_\beta$-Ramsey} if $\kappa\notin\R^\alpha(\Pi^1_\beta(\kappa))$. We show that, even though $\kappa$ being $\alpha$-$\Pi^1_\beta$-Ramsey may be equivalent to $\kappa$ being $\alpha$-$\Pi^1_{\beta'}$-Ramsey for some $\beta<\beta'<\kappa$ (see Theorem \ref{theorem_culmination} below), by choosing $\beta$'s appropriately, hypotheses of the form ``$\exists\kappa$ ($\kappa$ is $\alpha$-$\Pi^1_\beta$-Ramsey)'' for $\beta<\kappa$ yield a strict hierarchy of hypotheses between ``$\exists\kappa$ ($\kappa$ is $\Pi_\alpha$-Ramsey)'' and ``$\exists\kappa$ ($\kappa$ is $\Pi_{\alpha+1}$-Ramsey)''. In order to prove this hierarchy result, it seems that a careful analysis of the corresponding ideals is required (see Remark \ref{remark_hierarchy} and \Figure \ref{figure_a_refinement_of_the_ramsey_hierarchy} below). This seems to be a natural sequel to Feng's work, given that he included $\R^n(\NS_\kappa)$ in his hierarchy and when $\kappa$ is inaccessible the $\Pi^1_0$-indescribability ideal $\Pi^1_0(\kappa)$ equals $\NS_\kappa$. 

As a first observation, it is not hard to see that the ideals $\R^n(\Pi^1_1(\kappa))$ for $n<\omega$ fit into Feng's increasing chain (\ref{feng}) as expected:
\begin{align*}
[\kappa]^{<\kappa}\subseteq\NS_\kappa\subseteq\Pi^1_1(\kappa)\subseteq\R([\kappa]^{<\kappa})&\subseteq\R(\NS_\kappa)\subseteq\R(\Pi^1_1(\kappa))\subseteq\\
	&\R^2([\kappa]^{<\kappa})\subseteq\R^2(\NS_\kappa)\subseteq\R^2(\Pi^1_1(\kappa))\subseteq\cdots.
\end{align*}
However, since the Ramseyness of a cardinal $\kappa$ can be expressed by a $\Pi^1_2$ sentence over $V_\kappa$, it follows that the least Ramsey cardinal is not $\Pi^1_2$-indescribable and hence it is not true in general that $\Pi^1_2(\kappa)\subseteq\R([\kappa]^{<\kappa})$. We give a complete account of the nonlinear structure consisting of ideals $\R^\alpha(\Pi^1_\beta(\kappa))$ for $\alpha,\beta<\kappa$ under the containment relations $\subseteq$ and $\subsetneq$. 

After reviewing the relevant results of Baumgatner, Feng and Bagaria in Section \ref{section_preliminaries} and after establishing some basic properties of the ideals $\R^\alpha(\Pi^1_\beta(\kappa))$ in Section \ref{section_basic_properties}, we prove our first reflection result in Section \ref{section_a_first_reflection_result} concerning the ideals $\R^\alpha(\Pi^1_\beta(\kappa))$ for $\alpha,\beta<\kappa$. It follows from a result of Baumgartner \cite[Theorem 4.1]{MR0384553} that if $\kappa$ is a Ramsey cardinal then the set of cardinals less than $\kappa$ which are $\Pi^1_n$-indescribable for all $n$ is in the Ramsey filter on $\kappa$. We generalize this result by proving that for all $\alpha<\kappa$, if $\kappa\notin\R^{\alpha+1}([\kappa]^{<\kappa})$ (i.e. $\kappa$ is $\Pi_{\alpha+1}$-Ramsey in Feng's terminology) then the set of $\xi<\kappa$ such that $\xi\notin\R^\alpha(\Pi^1_\beta(\xi))$ \emph{for all} $\beta<\xi$ is in the filter dual to $\R^{\alpha+1}([\kappa]^{<\kappa})$. Hence ``$\exists\kappa$($\kappa\notin\R^{\alpha+1}([\kappa]^{<\kappa})$)'' is strictly stronger than ``$\exists\kappa$ $\forall \beta<\kappa$ ($\kappa\notin\R^\alpha(\Pi^1_\beta(\kappa))$)''. 

In Section \ref{section_describing_ramseyness}, we prove a technical lemma which is fundamental for the rest of the paper and which establishes an ordinal $\gamma(\alpha,\beta)$ which suffices to express the fact that a set $S\subseteq\kappa$ is in $\R^\alpha(\Pi^1_\beta(\kappa))^+$ using a $\Pi^1_{\gamma(\alpha,\beta)}$ sentence over $V_\kappa$.  This lemma provides a generalization of a result of Sharpe and Welch \cite[Remark 3.17]{MR2817562} which states that ``$S\in\R^\alpha([\kappa]^{<\kappa})^+$'' is a $\Pi^1_{2\cdot(1+\alpha)}$ property. 





In Section \ref{section_indescribability_in_finite_ramseyness}, we give a full account of the nonlinear containment structure of the ideals $\R^m(\Pi^1_\beta(\kappa))$ for $m\leq\omega$ and $\beta<\kappa$ (see \Figure \ref{figure_finite_ideal_diagram} below).  We derive several corollaries from this result. For example, we provide characterizations of the large cardinal property $\kappa\notin\R^m(\Pi^1_\beta(\kappa))$ which are analogous to Baumgartner's characterization of Ramseyness discussed above. As a consequence, $\kappa\notin\R^m(\Pi^1_\beta(\kappa))$ implies $\kappa$ is $\Pi^1_{\beta+2m}$-indescribable, and moreover the ideal $\R^m(\Pi^1_\beta(\kappa))$ equals the ideal generated by the $\Pi^1_{\beta+2m}$-indescribability ideal and a generalization of the pre-Ramsey ideal (see Corollary \ref{corollary_indescribability_in_finite_ramseyness} below). Furthermore, we prove that ``$\exists\kappa$($\kappa\notin\R^m(\Pi^1_{\beta+1}(\kappa))$'' is strictly stronger than ``$\exists\kappa$($\kappa\notin\R^m(\Pi^1_\beta(\kappa))$'' and that the large cardinal axioms associated to the ideals $\R^m(\Pi^1_\beta(\kappa))$ fit into a linear strict hierarchy when the ideals are nontrivial. Furthermore, in analogy with the fact quoted in Remark \ref{remark_fengs_defintion} above, we show that if $\kappa\notin\R^\omega(\Pi^1_\beta(\kappa))$ then for all $n<\omega$,
\[\R^\omega(\Pi^1_\beta(\kappa))=\R^\omega(\Pi^1_{\beta+n}(\kappa)).\]
Let us point out that the proof of this result is substantially different from the observations made in Remark \ref{remark_fengs_defintion} since the relevant ideals 
\[\{\R^m(\Pi^1_{\beta+n}(\kappa))\st \text{$m,n<\omega$ and $\beta<\kappa$}\}\] do not form an increasing chain. Another way of phrasing this result is that at the $\omega$-th level of the Ramsey hierarchy, the ideal chains $\<\R^\alpha(\Pi^1_\beta(\kappa))\st\alpha<\kappa\>$ and $\<\R^\alpha(\Pi^1_{\beta+n}(\kappa))\st\alpha<\kappa\>$ become equal. 

In Section \ref{section_indescribability_in_infinite_ramseyness}, we extend these results to the ideals $\R^\alpha(\Pi^1_\beta(\kappa))$ for $\omega<\alpha<\kappa$ and $\beta\in\{-1\}\cup\kappa$. That is, we provide a complete account of the containment relationships between ideals of the form $\R^\alpha(\Pi^1_\beta(\kappa))$. As a culmination of these results, given $\beta_0<\beta_1$ in $\{-1\}\cup\kappa$ we isolate the precise location in the Ramsey hierarchy at which the ideal chains $\<\R^\alpha(\Pi^1_{\beta_0}(\kappa))\st\alpha<\kappa\>$ and $\<\R^\alpha(\Pi^1_{\beta_1}(\kappa))\st\alpha<\kappa\>$ become equal by proving the following theorem (see \Figure \ref{figure_culmination} below for an illustration of this result). In what follows, $\Pi^1_{-1}(\kappa)=[\kappa]^{<\kappa}$ and $\Pi^1_0(\kappa)=\NS_\kappa$.

\begin{theorem}\label{theorem_culmination}
Suppose $\beta_0<\beta_1$ are in $\{-1\}\cup\kappa$ and let $\sigma=\ot(\beta_1\setminus\beta_0)$. Define $\alpha=\sigma\cdot\omega$. Suppose $\kappa\in\R^\alpha(\Pi^1_{\beta_1}(\kappa))^+$ so that the ideals under consideration are nontrivial. Then $\alpha$ is the least ordinal such that $\R^\alpha(\Pi^1_{\beta_0}(\kappa))=\R^\alpha(\Pi^1_{\beta_1}(\kappa))$.
\end{theorem}
\noindent Furthermore, we prove that the large cardinal hypotheses of the form ``$\exists \kappa$ $\kappa\notin\R^\alpha(\Pi^1_\beta(\kappa))$'' provide a strict linear refinement of Feng's original hierarchy up to $\Pi_\kappa$-Ramseyness (see Theorem \ref{theorem_hierarchy_result_for_infinite_alpha} and \Figure \ref{figure_a_refinement_of_the_ramsey_hierarchy} below).

Finally, in Section \ref{section_generic_embeddings}, as an application of our results we provide characterizations of the ideals $\R^\alpha(\Pi^1_\beta(\kappa))$ for $\alpha,\beta<\kappa$ in terms of generic elementary embeddings. As a special case, this also yields generic embedding characterizations of $\Pi^1_\xi$-indescribability and Ramseyness.

\section{Preliminaries}\label{section_preliminaries}

Here we describe some notation that will be used throughout the paper and some results from the literature. We cover some results of Baumgartner (from \cite{MR0384553} and \cite{MR0540770}) and Feng \cite{MR1077260} which serve as motivation for our results. Then we give a brief account of Bagaria's extension \cite{MR3894041} of $\Pi^1_n$-indescribability to $\Pi^1_\xi$-indescribability where $\xi$ can be any ordinal.

\subsection{Definitions and Notation}

Given an ideal $I$ on a cardinal $\kappa$ we let 
\[I^+=\{X\subseteq\kappa\st X\notin I\}\]
be the corresponding collection of positive sets and we let
\[I^*=\{X\subseteq\kappa\st \kappa\setminus S\in I\}\]
be the filter dual to $I$. For notational convenience, and in order to avoid double negations, in what follows we will often write $X\in I^+$ instead of $X\notin I$. If $\mathcal{A}\subseteq P(\kappa)$ is a collection of subsets of $\kappa$ then we write $\overline{\mathcal{A}}$ to denote the ideal on $\kappa$ generated by $\mathcal{A}$:
\[\overline{\mathcal{A}}=\{X\subseteq\kappa\st (\exists \mathcal{B}\in[\mathcal{A}]^{<\omega}) X\subseteq \bigcup\mathcal{B}\}.\]
An ideal $I$ on $\kappa$ is called \emph{nontrivial} if $I\neq P(\kappa)$

We will be concerned with showing that certain large cardinal ideals are obtained by taking the ideal generated by a union of some other large cardinal ideals. We will make repeated use of the following simple remark, which was used implicitly by Baumgartner (see the proof of Theorem 4.4 in \cite{MR0540770}).

\begin{remark}\label{remark_ideal_generated}
Suppose $I_0$, $I_1$ and $J$ are ideals on $\kappa$. In order to prove that $J=\overline{I_0\cup I_1}$, part of what we must show is that $J\supseteq\overline{I_0\cup I_1}$, or in other words $J^+\subseteq\overline{I_0\cup I_1}^+$. Notice that we may obtain a chain of equivalences directly from the definitions involved:
\begin{align*}
J^+\subseteq\overline{I_0\cup I_1}^+ &\iff \overline{I_0\cup I_1}\subseteq J \\
	&\iff I_0\cup I_1\subseteq J\\
	&\iff J^+\subseteq I_0^+\cap I_1^+.
\end{align*}
In what follows, in order to prove that the property $J^+\subseteq\overline{I_0\cup I_1}^+$ (or equivalently the property $J\supseteq\overline{I_0\cup I_1}$) holds for various ideals, we will prove $J^+\subseteq I_0^+\cap I_1^+$ and include a reference to this remark.
\end{remark}


\subsection{Baumgartner's ineffability hierarchy}\label{section_baumgartners_ineffability_hierarchy}

Let us review a few results due to Baumgartner using slightly different notation than \cite{MR0384553} and \cite{MR0540770}. Suppose $\kappa>\omega$ is a cardinal and $S\subseteq\kappa$. We say that $\vec{S}=\<S_\alpha\st\alpha\in S\>$ is a \emph{$(1,S)$-sequence}\footnote{Such sequences are sometimes called \emph{$S$-lists} (see \cite{MR2959668} or \cite{MR3913154}). However, we prefer Baumgartner's terminology because we will need to distinguish $(1,S)$-sequences from Feng's $(\omega,S)$-sequences (see Section \ref{section_feng}).} if for each $\alpha\in S$ we have $S_\alpha\subseteq\alpha$. Given a $(1,S)$-sequence $\vec{S}=\<S_\alpha\st\alpha\in S\>$, a set $H\subseteq S$ is \emph{homogeneous} for $\vec{S}$ if for all $\alpha,\beta\in H$ with $\alpha<\beta$ we have $S_\alpha=S_\beta\cap\alpha$. 

Given an ideal $I\supseteq[\kappa]^{<\kappa}$ on $\kappa$ we define another ideal $\I(I)$ by letting $S\in\I(I)^+$ if and only if for every $(1,S)$-sequence $\vec{S}$ there is a set $H\subseteq S$ in $I^+$ such that $H$ is homogeneous for $\vec{S}$. A set $S\subseteq\kappa$ is \emph{ineffable} if $S\in\I(\NS_\kappa)^+$. Baugartner showed that when $\kappa$ is an ineffable cardinal the collection $\I(\NS_\kappa)$ of non-ineffable subsets of $\kappa$ is a normal ideal on $\kappa$, which we call the \emph{ineffability ideal} on $\kappa$. Notice that $\I$ can be viewed as a function mapping ideals to ideals, which we call the \emph{ineffability operator}.

Baumgartner \cite[Theorem 5.3]{MR0384553} gave several characterizations of ineffability in terms of partition properties. Given a set $S\subseteq\kappa$, a function $f:[S]^2\to\kappa$ is said to be \emph{regressive} if $f(a)<\min a$ for all $a\in[S]^2$. A set $H\subseteq S$ is \emph{homogeneous} for a function $f:[S]^2\to\kappa$ if $f\restrict [H]^2$ is constant. 

\begin{theorem}[Baumgartner]\label{theorem_baumgartner_ineff_char}
Let $\kappa$ be a cardinal and $S\subseteq\kappa$. The following are equivalent.
\begin{enumerate}
\item $S$ is ineffable.
\item For every regressive function $f:[S]^2\to\kappa$ there is a set $H\subseteq S$ stationary in $\kappa$ which is homogeneous for $f$.
\item $\kappa$ is regular and for every function $f:[S]^2\to 2$ there is a set $H\subseteq S$ stationary in $\kappa$ which is homogeneous for $f$.
\end{enumerate}
\end{theorem}

Now suppose that $\vec{I}=\<I_\alpha\st\text{$\alpha\leq\kappa$ is an uncountable cardinal}\>$ is a sequence such that each $I_\alpha\supseteq[\alpha]^{<\alpha}$ is an ideal on $\alpha$ and $I_\alpha=P(\alpha)$ when $\alpha$ is a singular cardinal. We define an ideal $\I_0(\vec{I})$ on $\kappa$ by letting $S\in\I_0(\vec{I})^+$ if and only if for every $(1,S)$-sequence $\vec{S}$ and every club $C\subseteq\kappa$ there is an $\alpha\in S\cap C$ for which there is a set $H\subseteq S\cap C\cap\alpha$ in $I_\alpha^+$ homogeneous for $\vec{S}$. When no confusion will arise, as in the case where the nontrivial ideals $I_\alpha$ have a uniform definition, we write $\I_0(I_\kappa)$ instead of $\I_0(\vec{I})$. 

For example, let $\vec{J}=\<J_\alpha\st\text{$\alpha\leq\kappa$ is a cardinal}\>$ be defined by letting $J_\alpha=\NS_\alpha$ when $\alpha$ is regular and $J_\alpha=P(\alpha)$ when $\alpha$ is singular. A set $S\subseteq\kappa$ is \emph{subtle} if $S\in\I_0(\NS_\kappa)^+=\I_0(\vec{J})^+$.\footnote{Baumgartner showed that $S\in\I_0(\NS_\kappa)^+$ is equivalent to the more often used definition of subtlety of a set $S$ given in Theorem \ref{theorem_subtle} (4). We use the stated definition of subtlety of $S$ for ease of presentation.} Futhermore, Baumgartner proved that if $\kappa$ is a subtle cardinal then $\I_0(\NS_\kappa)$ is a normal ideal on $\kappa$, which we call the \emph{subtle ideal} on $\kappa$. We refer to $\I_0$ as the \emph{subtle operator}. Recall that every ineffable set is subtle, the least subtle cardinal is not $\Pi^1_1$-indescribable and, as shown by Baumgartner \cite[Theorem 4.1]{MR0384553}, the existence of a subtle cardinal is strictly stronger than the existence of a cardinal which is $\Pi^1_n$-indescribable for all $n<\omega$. 

As another example, let $\vec{J}=\<J_\alpha\st\text{$\alpha\leq\kappa$ is a cardinal}\>$ be a sequence of ideals defined by letting $J_\alpha=\Pi^1_1(\alpha)$ when $\alpha$ is $\Pi^1_1$-indescribable and $J_\alpha=P(\alpha)$ otherwise. Then $\I_0(\Pi^1_1(\kappa))=\I_0(\vec{J})$ and a set $S$ is in $\I_0(\Pi^1_1(\kappa))^+$ if and only if for every $(1,S)$-sequence $\vec{S}$ and every club $C\subseteq\kappa$ there is an $\alpha\in S\cap C$ for which there is a set $H\subseteq S\cap C\cap \alpha$ in $\Pi^1_1(\alpha)^+$which is homogeneous for $\vec{S}$.\footnote{Note that the set $H$ being in $\Pi^1_1(\alpha)^+$ implies $\alpha$ is $\Pi^1_1$-indescribable.}

Baumgartner showed \cite[Theorem 5.1]{MR0384553} that subtlety can be characterized using partition properties.

\begin{theorem}[Baumgartner]\label{theorem_subtle}
Let $\kappa$ be a cardinal and $S\subseteq\kappa$. The following are equivalent.
\begin{enumerate}
\item $S$ is subtle, that is, $S\in\I_0(\NS_\kappa)^+$.
\item For every regressive function $f:[S]^2\to\kappa$ and every club $C\subseteq\kappa$ there is a regular cardinal $\alpha\leq\kappa$ and a set $H\subseteq S\cap C\cap\alpha$ stationary in $\alpha$ which is homogeneous for $f$.
\item For every function $f:[S]^2\to 2$ and every club $C\subseteq\kappa$ there is a regular cardinal $\alpha\leq\kappa$ and a set $H\subseteq S\cap C\cap \alpha$ stationary in $\alpha$ which is homogeneous for $f$.
\item For every $(1,S)$-sequence $\vec{S}$ and every club $C\subseteq\kappa$ there is a set $\{\alpha_0,\alpha_1\}\in[S\cap C]^2$ which is homogeneous for $\vec{S}$.
\end{enumerate}
\end{theorem}

The following theorem, perhaps one of the most noteworthy of \cite{MR0384553}, shows that in order to have a full understanding of certain large cardinals, one \emph{must} consider large cardinal ideals. Taking $n=0$ in the following theorem, one can easily see that a cardinal $\kappa$ is ineffable if and only if it is subtle, $\Pi^1_2$-indescribable and additionally the subtle ideal and the $\Pi^1_2$-indescribable ideal generate a nontrivial ideal which equals the ineffability ideal; moreover, reference to these ideals cannot be removed from this characterization.

\begin{theorem}[Baumgartner]\label{theorem_baumgarnter_ineffability}
Suppose $\kappa$ is a cardinal and $n<\omega$. Then $\kappa\in\I(\Pi^1_n(\kappa))^+$ if and only if both of the following hold.
\begin{enumerate}
\item $\kappa\in\I_0(\Pi^1_n(\kappa))^+$ and $\kappa\in \Pi^1_{n+2}(\kappa)^+$.
\item The ideal generated by $\I_0(\Pi^1_n(\kappa))\cup\Pi^1_{n+2}(\kappa)$ is nontrivial and equals $\I(\Pi^1_n(\kappa))$.
\end{enumerate}
Moreover, reference to the ideals in the above characterization cannot be removed because the least cardinal $\kappa$ such that $\kappa\in\I_0(\Pi^1_n(\kappa))^+$ and $\kappa\in\Pi^1_{n+2}(\kappa)^+$ is not in $\I(\Pi^1_n(\kappa))^+$.
\end{theorem}

In his second article \cite{MR0540770} on ineffability properties, Baumgartner iterated the ineffability operator $\I$ and defined an increasing chain of ideals as follows. Define $\I^0(\NS_\kappa)=\NS_\kappa$ and $\I^{\alpha+1}(\NS_\kappa)=\I(\I^\alpha(\NS_\kappa))$. If $\alpha$ is a limit ordinal let $I^\alpha(\NS_\kappa)=\bigcup_{\xi<\alpha}\I^\xi(\NS_\kappa)$. Since the ideals $\I^\alpha(\NS_\kappa)$ form an increasing chain and there are only $2^\kappa$ subsets of $\kappa$, there must be an $\alpha<(2^\kappa)^+$ such that $\I^\alpha(\NS_\kappa)=\I^{\alpha+1}(\NS_\kappa)$. A cardinal $\kappa$ is \emph{completely ineffable} if when $\alpha$ is the least ordinal such that $\I^\alpha(\NS_\kappa)=\I^{\alpha+1}(\NS_\kappa)$ the ideal $\I^\alpha(\NS_\kappa)$ is nontrivial. Baumgartner introduced \emph{canonical function} in order to prove \cite[Theorem 3.7]{MR0540770} that if $\beta<\kappa^+$ and $\kappa\in\I^\beta(\NS_\kappa)^+$ (i.e. $\I^\beta(\NS_\kappa)$ is a nontrivial ideal) then for all $\alpha<\beta$ the containment $\I^\alpha(\NS_\kappa)\subsetneq\I^{\alpha+1}(\NS_\kappa)$ is proper.

\begin{remark}
Although Baumgartner briefly mentions the ideals $\I^m(\Pi^1_n(\kappa))$ in \cite{MR0540770} (see the discussion following Corollary 3.5), they, as well as ideals of the form $\I^\alpha(\Pi^1_n(\kappa))$ for ordinals $\alpha>1$, seem to be otherwise absent from both \cite{MR0384553} and \cite{MR0540770}.
\end{remark}

\subsection{Baumgartner's result on the Ramsey ideal and Feng'€™s Ramsey hierarchy}\label{section_feng}


Recall from Section \ref{section_introduction}, that Feng \cite{MR1077260} defined the Ramsey operator $\R$, which is analogous to the ineffability operator $\I$, and iterated $\R$ in order to define completely Ramsey cardinals. 
As noted by Feng \cite[Definition 2.1]{MR1077260}, it is easy to see that if $I\subseteq J$ are ideals on $\kappa$ then $\R(I)\subseteq\R(J)$. Although it will not be used in what follows, let us show that under reasonable assumptions one obtains proper containment $\R(I)\subsetneq\R(J)$.

\begin{proposition}
If $[\kappa]^{<\kappa}\subseteq I\subsetneq J$ are nontrivial ideals on $\kappa$ such that $\R(I)^+\cap J\neq\emptyset$ and for all $X\in I^+$ there exist $X_0,X_1\in I^+$ such that $X=X_0\sqcup X_1$, then $\R(I)\subsetneq\R(J)$.
\end{proposition}

\begin{proof} 
Suppose $I\subsetneq J$ are as in the statement of the lemma. Let us show that $\R(I)\neq\R(J)$. Choose $X\in \R(I)^+\cap J$. Since $X\in I^+$, there exist $X_0,X_1\in I^+$ such that $X=X_0\sqcup X_1$. Notice that $X_0,X_1\in J$ since $X_0,X_1\subseteq X\in J$. Define a function $f:[X]^{<\omega}\to 2$ by
\[
f(\vec{\alpha})=
\begin{cases}
1 & \textrm{if $\vec{\alpha}\in X_0^{<\omega}\cup X_1^{<\omega}$}\\
0 & \textrm{otherwise}
\end{cases}
\]
Since $X\in\R(I)^+$ there is a homogeneous set for $f$ in $I^+$. However, it is straightforward to show that if $H\subseteq X$ is homogeneous for $f$, then $H$ is either a subset of $X_0$ or a subset of $X_1$, and is therefore in $J$. Thus $X\in \R(J)$.
\end{proof}

Let us define another operator $\R_0$ which is analogous to $\I_0$. Suppose that $\vec{I}=\<I_\alpha\st\text{$\alpha\leq\kappa$ is a cardinal}\>$ is a sequence such that each $I_\alpha\supseteq[\alpha]^{<\alpha}$ is an ideal on $\alpha$. Recall that for a set $S\subseteq\kappa$ a function $f:[S]^{<\omega}\to\kappa$ is \emph{regressive} if $f(a)<\min(a)$ for all $a\in[S]^{<\omega}$. We define an ideal $\R_0(\vec{I})$ on $\kappa$ by letting $S\in\R_0(\vec{I})^+$ if and only if for every regressive function $f:[S]^{<\omega}\to\kappa$ and every club $C\subseteq\kappa$ there is an $\alpha\in S\cap C$ for which there is a set $H\subseteq S\cap C\cap\alpha$ in $I_\alpha^+$ which is homogeneous for $f$, meaning that $f\restrict[H]^n$ is constant for each $n<\omega$. As before (see the discussion after Theorem \ref{theorem_baumgartner_ineff_char}), many of the ideals $I_\alpha$ will be understood to be trivial, and when no confusion will arise, as in the case where the ideals $I_\alpha$ have a uniform definition, we write $\R_0(I_\kappa)$ instead of $\R_0(\vec{I})$. Baumgartner defined a set $S\subseteq\kappa$ to be \emph{pre-Ramsey} if and only if $S\in\R_0([\kappa]^{<\kappa})^+$. Thus, pre-Ramseyness is to Ramseyness as subtlety is to ineffability.


Feng \cite[Theorem 2.3]{MR1077260} gave a characterization of Ramseyness which resembles the definition of ineffability. Suppose $S\subseteq\kappa$. For each $n<\omega$ and for all increasing sequences $\alpha_1<\cdots<\alpha_n$ taken from $S$ suppose that $S_{\alpha_1,\ldots,\alpha_n}\subseteq\alpha_1$. Then we say that
\[\vec{S}=\<S_{\alpha_1,\ldots,\alpha_n}\st n<\omega\land (\alpha_1,\ldots,\alpha_n)\in[S]^n\>\]
is an $(\omega,S)$-sequence. A set $H\subseteq S$ is said to be \emph{homogeneous} for an $(\omega,S)$-sequences $\vec{S}$ if for all $0<n<\omega$ and for all increasing sequences $\alpha_1<\cdots<\alpha_n$ and $\beta_1<\cdots<\beta_n$ taken from $H$ with $\alpha_1\leq\beta_1$ we have $S_{\alpha_1\cdots\alpha_n}=S_{\beta_1\cdots\beta_n}\cap\alpha_1$.


The following theorem is essentially due to Feng (see \cite[Theorem 2.2 and Theorem 2.3]{MR1077260}), with the exception of clause (4) which appears in \cite[Theorem 3.2]{MR2817562}.



\begin{theorem}[Feng]\label{theorem_omega_S}
Let $\kappa$ be a regular cardinal and suppose $I$ is an ideal on $\kappa$ such that $I\supseteq \NS_\kappa$. For $S\subseteq \kappa$ the following are equivalent.
\begin{enumerate}
\item Every function $f:[S]^{<\omega}\to 2$ has a homogeneous set $H\in P(S)\cap I^+$.
\item For all $\gamma<\kappa$, every function $f:[S]^{<\omega}\to\gamma$ has a homogeneous set $H\in P(S)\cap I^+$.
\item Every structure $\mathcal{A}$ in a language of size less than $\kappa$ with $\kappa\subseteq\mathcal{A}$ has a set of indiscernibles $H\in P(S)\cap I^+$.
\item $S\in\R(I)^+$, that is, for every regressive function $f:[S]^{<\omega}\to\kappa$ there is a set $H\in P(S)\cap I^+$ which is homogeneous for $f$.
\item For every club $C\subseteq\kappa$, every regressive function $f:[S]^{<\omega}\to\kappa$ has a homogeneous set $H\in P(S\cap C)\cap I^+$.
\item For all $(\omega,S)$-sequences $\vec{S}$ there is a set $H\in P(S)\cap I^+$ which is homogeneous for $\vec{S}$.
\end{enumerate}
\end{theorem}

\begin{proposition}\label{proposition_4_5_6}
Suppose $I\supseteq[\kappa]^{<\kappa}$ is an ideal on a regular cardinal $\kappa$. Then clauses (4), (5) and (6) of Theorem \ref{theorem_omega_S} are equivalent.
\end{proposition}

\begin{proof}
The equivalence of (5) and (6) is due to Feng \cite[Theorem 2.3]{MR1077260}. It is easy to see that (5) implies (4). Let us show that (4) implies (5).\footnote{The author would like to thank the anonymous referee for this argument as well as Victoria Gitman for an earlier version.} Suppose (4) holds and fix a regressive function $f:[S]^{<\omega}\to\kappa$ and a club $C\subseteq\kappa$. First, let us argue that $|S\cap C|=\kappa$. Suppose $|S\cap C|<\kappa$ and let $\alpha=\sup(S\cap C)$. Define $g:S\setminus(\alpha+1)\to\kappa$ by letting $g(\xi)$ be the greatest element of $C$ which is less than $\xi$. Notice that $g$ is regressive. Now let $G:[S]^{<\omega}\to\kappa$ be any regressive function with $G(\{\xi\})=g(\xi)$ for $\xi\in S\setminus(\alpha+1)$. By (4), there is a homogeneous set $H\in P(S)\cap I^+$ for $G$, but then since $G\restrict[H]^1$ is constant, it follows that $g$ is constant on the unbounded set $H$, but this contradicts the fact that $C$ is unbounded. Thus $|S\cap C|=\kappa$. To prove (5) we must show that there is a homogeneous set $H\in P(S\cap C)\cap I^+$ for $f$. Let us define another function $h:[S\cap C]^{<\omega}\to \kappa$ by letting 
\[h(\{\xi\})=\begin{cases}
f(\{\xi\}) & \text{if $x\in S\cap C$}\\
\sup(C\cap\xi) & \text{if $x\in S\setminus C$}
\end{cases}
\]
and for $n>1$ let $h\restrict[S\cap C]^n=f\restrict[S\cap C]^n$. Then $h$ is regressive, and using (4), there is a homogeneous set $H\in P(S)\cap I^+$ for $h$. Using an argument similar to the above argument for $|S\cap C|=\kappa$, we see that $|H\cap (S\setminus C)|<\kappa$ and hence $H\cap C\in I^+$. Since $h\restrict[S\cap C]^{<\omega}=f\restrict[S\cap C]^{<\omega}$, it follows that $H\cap C$ is homogeneous for $f$, yielding (5).
\end{proof}

We will need the next easy consequence of Proposition \ref{proposition_4_5_6}.

\begin{corollary}\label{corollary_one_S}
Suppose $\kappa$ is a cardinal, $I\supseteq\NS_\kappa$ is an ideal on $\kappa$ and $S\in\R(I)^+$. Then every $(1,S)$-sequence $\vec{S}=\<S_\alpha\st\alpha\in S\>$ has a homogeneous set $H\subseteq S$ in $I^+$.
\end{corollary}

\begin{proof}
Let $\vec{S}_*$ be any $(\omega,S)$-sequence extending $\vec{S}$. Since $S\in\R(I)^+$ there is a set $H\subseteq S$ in $I^+$ which is homogeneous for $\vec{S}_*$. Clearly $H$ is also homogeneous for $\vec{S}$.
\end{proof}

In order to prove a certain reflection result (see Theorem \ref{theorem_ramsey_reflection} below), we will use a characterization of Ramsey sets which is given in terms of elementary embeddings; the analogue of this characterization for Ramsey cardinals is essentially due to Michell \cite{MR534574} and was further explored by Gitman \cite{MR2830415}. Recall that a transitive set $M\models \ZFC^-$ of size $\kappa$ with $\kappa\in M$ is called a \emph{weak $\kappa$-model}. If $M$ is a weak $\kappa$-model we say that an $M$-ultrafilter $U$ on $\kappa$ is \emph{weakly amenable} if for every sequence $\<X_\alpha\st\alpha<\kappa\>$ in $M$ of subsets of $\kappa$, the set $\{\alpha<\kappa\st X_\alpha\in U\}$ is an element of $M$. An $M$-ultrafilter $U$ on $\kappa$ is \emph{countably complete} if whenever $\<A_n\st n<\omega\>$ is a sequence of elements of $U$, possibly external to $M$, it follows that $\bigcap_{n<\omega}A_n\neq\emptyset$. Taking $S=\kappa$ in the following theorem one obtains Mitchell's characterization of Ramsey cardinals \cite[Proposition 2.8(3)]{MR2830415}.

\begin{theorem}[Mitchell]\label{theorem_ramsey_equiv}
Suppose $\kappa$ is a regular cardinal and $S\subseteq\kappa$. Then $S\in\R([\kappa]^{<\kappa})^+$ (i.e. $S$ is a Ramsey set) if and only if for every $A\subseteq\kappa$ there is a weak $\kappa$-model $M$ with $A,S\in M$ for which there exists a weakly amenable countably complete $M$-ultrafilter $U$ on $\kappa$ with $S\in U$.
\end{theorem}

\begin{proof} Since the proof of this theorem is a straightforward modification of arguments appearing in \cite{MR2710923} and \cite{MR2830415}, we provide a brief sketch together with citations to more detailed arguments.\footnote{The author would like to thank Victoria Gitman for a helpful conversation regarding this argument.}


Suppose $S$ is Ramsey. Fix $A\subseteq\kappa$. We follow the argument in \cite[Section 4]{MR2830415}, which is also given in more detail in \cite[Chapter 2]{MR2710923}. First we argue that there is a set $H\in [\kappa]^\kappa$ of good indiscernibles for $(L_\kappa[A,S],A,S)$ (recall that a set of indiscernibles $H\subseteq\kappa$ for $(L_\kappa[A,S],A,S)$ is \emph{good} if for all $\gamma\in I$, $\gamma$ is a cardinal, $(L_\gamma[A,S],A,S)\prec(L_\kappa[A,S],A,S)$ and $I\setminus\gamma$ is a set of indiscernibles for $(L_\kappa[A,S],A,S,\xi)_{\xi\in \gamma}$). Using the argument for \cite[Lemma 2.43]{MR2710923}, it follows that there is a club $C\subseteq\kappa$ and a regressive function $h:[C]^{<\omega}\to\kappa$ such that any homogeneous set for $h$ is a set of good indiscernibles for $(L_\kappa[A,S],A,S)$. Since $S$ is Ramsey, it follows that $S\cap C$ is Ramsey, and thus the regressive function $h:[S\cap C]^{<\omega}\to \kappa$ has a homogeneous set $H\in P(S\cap C)\cap[\kappa]^\kappa$. Using the fact that $H$ is a good set of indiscernibles for $(L_\kappa[A,S],A,S)$, and by following the argument for \cite[Theorem 2.35, pp. 104--114]{MR2710923}, one can construct a weak $\kappa$-model $M$ with $A,S\in M$ for which there is a weakly amenable countably complete $M$-ultrafilter $U$ on $\kappa$ such that a set $X\in P(\kappa)^M$ is in $U$ if and only if there exists $\alpha<\kappa$ with $H\setminus \alpha\subseteq X$ (see \cite[Lemma 2.44.12]{MR2710923} for this characterization of $U$). Since $H\subseteq S\cap C$, it follows that $S\in U$.

Conversely, suppose that for every $A\subseteq\kappa$ there is a weak $\kappa$-model $M$ with $A,S\in M$ for which there is a weakly amenable countably complete $M$-ultrafilter $U$ on $\kappa$ with $S\in U$. To see that $S$ is Ramsey, fix a regressive function $f:[S]^{<\omega}\to\kappa$. We follow the proof of \cite[Theorem 3.10]{MR2830415}. Let $M$ be a weak $\kappa$-model with $f\in M$ and let $U$ be a weakly amenable countably complete $M$-ultrafilter with $S\in U$. Using the weak amenability of $U$ we can define the product ultrafilters $U^n$ on $P(\kappa^n)^M$ for all $n<\omega$ as follows. Let $U^0=U$. Given $U^n$ let $U^{n+1}$ be the ultrafilter on $\kappa^n\times\kappa$ defined by $X\in U^{n+1}$ if and only if $X\in P(\kappa^n\times\kappa)^M$ and $\{\vec{\alpha}\in \kappa^n\st \{\xi<\kappa\st \vec{\alpha}\concat \xi\in X\}\in U\}\in U^n$. Since $S\in U$, it follows by induction that $S^n\in U^n$ for all $n<\omega$. For each $n<\omega$, let $j_{U^n}:M\to N_n$ be the ultrapower of $M$ by $U^n$ and let $f_n=f\restrict[S]^n$. Recall that for each $n<\omega$, the critical point of $j_{U^n}$ is $\kappa$ and $X\in U^n$ if and only if $(\kappa,j_U(\kappa),\ldots,j_{U^{n-1}}(\kappa))\in j_{U^n}(X)$ (see \cite[Lemma 2.33]{MR2710923}). Fix $n<\omega$. By elementarity $j_{U^n}(f_n)$ is regressive and hence $j_{U^n}(f_n)(\kappa,j_U(\kappa),\ldots,j_{U^{n-1}}(\kappa))=\eta<\kappa$. Thus $H_n'=\{\vec{\alpha}\in S^n\st f_n(\vec{\alpha})=\eta\}\in U^n$. By \cite[Proposition 2.5]{MR2830415}, there is a set $H_n\in U$ such that for all increasing sequences $\alpha_1<\cdots<\alpha_n$ from $H_n$ we have $(\alpha_1,\ldots,\alpha_n)\in H_n'$. Thus, $H_n$ is homogeneous for $f_n$. Notice that $H=\bigcap_{n<\omega} H_n$ must have size $\kappa$ because if $H$ were bounded, say $H\subseteq\alpha<\kappa$, then the set $(\kappa\setminus\alpha)\cap \bigcap_{n<\omega}H_n$ would be empty, which is impossible by countable completeness since $\kappa\setminus\alpha\in U$. Since $H$ is homogeneous for $f$ we have established $S\in\R([\kappa]^{<\kappa})$.
\end{proof}

Theorem \ref{theorem_ramsey_equiv} naturally leads to the following characterization of $\R([\kappa]^{<\kappa})$ due to Mitchell in terms of elementary embeddings (see \cite{MR2830415} for more information).

\begin{theorem}[Mitchell]\label{theorem_gitman}
A set $S\subseteq\kappa$ is Ramsey, or, in other words, $S\in\R([\kappa]^{<\kappa})^+$, if and only if for every $A\subseteq\kappa$ there is a weak $\kappa$-model $M$ with $A,S\in M$ and there is an elementary embedding $j:M\to N$ such that
\begin{enumerate}
\item The critical point of $j$ is $\kappa$.
\item $N$ is transitive.
\item $P(\kappa)^M=P(\kappa)^N$
\item Whenever $\<A_n\st n<\omega\>$ is a sequence of elements of $P(\kappa)^M$ which is possibly external to $M$ and $\kappa\in j(A_n)$ for all $n<\omega$, then $\bigcap_{n<\omega}A_n\neq\emptyset$.
\item $\kappa\in j(S)$.
\end{enumerate}
\end{theorem}

Baumgartner \cite[Theorem 4.4]{MR0540770} gave a characterization of Ramsey cardinals which is similar to his characterization of ineffable cardinals (Theorem \ref{theorem_baumgarnter_ineffability} above): $\kappa$ is Ramsey if and only if it is pre-Ramsey, $\Pi^1_1$-indescribable and additionally the pre-Ramsey ideal and the $\Pi^1_1$-indescribability ideal generate a nontrivial ideal which equals the Ramsey ideal; moreover, reference to these ideals cannot be removed from this characterization. Feng \cite[Theorem 4.8]{MR1077260} generalized Baumgartner's characterization of Ramseyness. Taking $m=1$ and $n=0$ in the following theorem yields Baumgartner's result.

\begin{theorem}[Feng]
Suppose $\kappa$ is a cardinal and let $I_{-1}=[\kappa]^{<\kappa}$ and $I_0=\NS_\kappa$. Let $1\leq m<\omega$ and $n\in\{-1,0\}$. Then $\kappa\in\R^m(I_n)^+$ if and only if both of the following hold.
\begin{enumerate}
\item $\kappa\in\R_0(\R^{m-1}(I_n))^+$ and $\kappa\in \Pi^1_{n+2m}(\kappa)^+$.
\item The ideal generated by $\R_0(\R^{m-1}(I_n))\cup\Pi^1_{n+2m}(\kappa)$ is nontrivial and equals $\R^m(I_n)$.
\end{enumerate}
Moreover, reference to the ideals in the above characterization cannot be removed because the least cardinal $\kappa$ such that $\kappa\in\R_0(\R^{m-1}(I_n))^+$ and $\kappa\in\Pi^1_{n+2m}(\kappa)^+$ is not in $\R(I_n)^+$.
\end{theorem}

Feng \cite[Theorem 5.2]{MR1077260} also proved that the $\Pi_\alpha$-Ramsey cardinals form a hierarchy which is strictly increasing in consistency strength.

\begin{theorem}[Feng]
Let $\<f_\alpha\st\alpha<\kappa^+\>$ be a sequence of canonical functions on a regular uncountable cardinal $\kappa$.\footnote{See \cite{MR0540770} and \cite{MR1077260} for the definition and relevant facts concerning canonical sequences of functions.} If $\kappa$ is $\Pi_{\alpha+1}$-Ramsey and $\alpha<\kappa^+$, then $\{\gamma<\kappa\st\text{$\gamma$ is $\Pi_{f_\alpha(\gamma)}$-Ramsey}\}$ is in the $\Pi_{\alpha+1}$-Ramsey filter on $\kappa$.
\end{theorem}

\subsection{Transfinite indescribability}\label{section_bagaria}

Sharpe and Welch \cite[Definition 3.21]{MR2817562} introduced a version of $\Pi^1_\xi$-indescribability for $\xi\geq\omega$ which is defined in terms of the existence of winning strategies in certain finite games. Inedpendently, Bagaria \cite[Section 4]{MR3894041} defined a natural notion of $\Pi^1_\xi$-formula for $\xi\geq\omega$ using infinitary logic, and then gave a definition of $\Pi^1_\xi$-indescribability in terms of rank initial segments of the set theoretic universe. As mentioned in Section \ref{section_introduction}, it is not difficult to see that Sharpe-Welch notion of the $\Pi^1_\xi$-indescribability of a cardinal $\kappa$ is equivalent to Bagaria's notion for $\xi<\kappa$. For the reader's convenience, we summarize Bagaria's definition, which we will use throughout the paper.

\begin{definition}[Bagaria]\label{definition_bagaria}
In what follows all quantifiers which are explicitly displayed are of second order. For any ordinal $\xi$, we say that a formula is $\Sigma^1_{\xi+1}$ if it is of the form 
\[\exists X_0,\ldots,X_k\varphi(X_0,\ldots,X_k)\]
where $\varphi(X_0,\ldots,X_k)$ is $\Pi^1_\xi$. And a formula is $\Pi^1_{\xi+1}$ if it is of the form 
\[\forall X_0,\ldots,X_k\varphi(X_0,\ldots,X_k)\]
where $\varphi(X_0,\ldots,X_k)$ is $\Sigma^1_\xi$.

If $\xi$ is a limit ordinal, we say that a formula is $\Pi^1_\xi$ if it is of the form
\[\bigwedge_{\zeta<\xi}\varphi_\zeta\]
where $\varphi_\zeta$ is $\Pi^1_\zeta$ for all $\zeta<\xi$ and the infinite conjunction has only finitely-many free second-order variables. And we say that a formula is $\Sigma^1_\xi$ if it is of the form 
\[\bigvee_{\zeta<\xi}\varphi_\zeta\]
where $\varphi_\zeta$ is $\Sigma^1_\zeta$ for all $\zeta<\xi$ and the infinite disjunction has only finitely-many free second-order variables.
\end{definition}


\begin{definition}[Bagaria]
Suppose $\kappa$ is a cardinal. A set $S\subseteq\kappa$ is \emph{$\Pi^1_\xi$-indescribable} if for all subsets $A\subseteq V_\kappa$ and every $\Pi^1_\xi$ sentence $\varphi$, if $(V_\kappa,\in,A)\models\varphi$ then there is some $\alpha\in S$ such that $(V_\alpha,\in,A\cap V_\alpha)\models\varphi$.
\end{definition}

\begin{remark}
As pointed out by Bagaria, it is clear from the definition that if $\kappa$ is $\Pi^1_\xi$-indescribable then $\xi<\kappa$. When we write $\kappa\in\Pi^1_\xi(\kappa)^+$, this indicates that $\kappa$ is $\Pi^1_\xi$-indescribable, and hence it should be understood that $\xi<\kappa$.
\end{remark}

There is a natural normal ideal associated to the $\Pi^1_\xi$-indescribability of a cardinal. The following result is due to Bagaria \cite{MR3894041}, and independently Brickhill and Welch (see \cite[Lemma 3.21]{BrickhillWelch} and \cite[Lemma 4.3.7]{Brickhill:Thesis}).


\begin{proposition}[Bagaria; Brickhill-Welch]
If $\kappa$ is a $\Pi^1_\xi$-indescribable cardinal where $\xi<\kappa$ then the collection
\[\Pi^1_\xi(\kappa)=\{X\subseteq\kappa\st\text{$X$ is not $\Pi^1_\xi$-indescribable}\}\]
is a nontrivial normal ideal on $\kappa$.
\end{proposition}



In some cases, it will be convenient to work with a weak version of $\Pi^1_\xi$-indescribability.

\begin{definition} A set $S\subseteq\kappa$ is \emph{weakly $\Pi^1_\xi$-indescribable} if for all $A\subseteq\kappa$ and all $\Pi^1_\xi$ sentences $\varphi$, if $(\kappa,\in,A)\models\varphi$ then there is an $\alpha\in S$ such that $(\alpha,\in,A\cap\alpha)\models\varphi$.
\end{definition}

\begin{remark} It is easy to check that if $\kappa$ is inaccessible a set $S\subseteq\kappa$ is weakly $\Pi^1_\xi$-indescribable if and only if it is $\Pi^1_\xi$-indescribable, and hence
\[\Pi^1_\xi(\kappa)=\{X\subseteq\kappa\st\text{$X$ is weakly $\Pi^1_\xi$-indescribable}\}.\]
\end{remark}

Next, we show that one of Baumgartner'€™s fundamental technical lemmas from \cite{MR0384553} can be extended from $\Pi^1_n$-indescribability to Bagaria'€™s notion of $\Pi^1_\xi$-indesc\-ribability. The following lemma also extends a result of Brickhill and Welch \cite[Theorem 5.3]{BrickhillWelch} concerning their notion of \emph{$\gamma$-ineffability}, where the $\gamma$-ineffability of $\kappa$ is equivalent under the assumption $V=L$ to $\kappa\in\I(\Pi^1_\gamma(\kappa))^+$ for $\gamma<\kappa$.

\begin{lemma}\label{lemma_baumgartner_bagaria}
Suppose $S\subseteq\kappa$ and for every $(1,S)$-sequence $\vec{S}=\<S_\alpha\st\alpha\in S\>$ there is a $B\in Q$ such that $B$ is homogeneous for $\vec{S}$. If $Q\subseteq \bigcap_{\xi\in\{-1\}\cup\beta}\Pi^1_\xi(\kappa)^+$ where $\beta<\kappa$, then $S$ is a $\Pi^1_{\beta+1}$-indescribable subset of $\kappa$. (Notice that if $\beta=\eta+1$ is a successor ordinal, the result states that if $Q\subseteq\Pi^1_\eta(\kappa)^+$ where $\eta<\kappa$, then $S$ is a $\Pi^1_{\eta+2}$-indescribable subset of $\kappa$.)
\end{lemma}

\begin{proof} The case in which $\beta<\omega$ is due to Baumgartner (see \cite[Lemma 7.1]{MR0384553} for the case in which $1\leq\beta<\omega$ and the discussion after Theorem 7.2 in \cite{MR0384553} for the case in which $\beta=0$). Notice that the assumption that every $(1,S)$-sequence $\vec{S}=\<S_\alpha\st\alpha\in S\>$ has a homogeneous set $H\in\bigcap_{\xi\in\{-1\}\cup\beta}\Pi^1_\xi(\kappa)^+$ implies that $\kappa$ is inaccessible, and hence, in order to show that $\kappa$ is $\Pi^1_{\beta+1}$-indescribable, it suffices to show that $\kappa$ is weakly $\Pi^1_\beta$-indescribable.
First, let us consider the case in which $\beta$ is a limit ordinal. Suppose $(\kappa,\in,A)\models\forall X\left(\bigvee_{\xi<\beta}\varphi_\xi\right)$ where $\varphi_\xi$ is $\Sigma^1_\xi$ for $\xi<\beta$. Suppose for all $\alpha\in S$ there is an $S_\alpha\subseteq \alpha$ such that $(\alpha,\in,A\cap \alpha)\models\bigwedge_{\xi<\beta}\lnot\varphi_\xi[S_\alpha]$. By assumption there is a $B\in Q$ which is homogeneous for $\vec{S}=\<S_\alpha\st\alpha\in S\>$. Let $X=\bigcup_{\alpha\in B}S_\alpha$. Then for some $\zeta<\beta$ we have $(\kappa,\in,A)\models\varphi_\zeta[X]$. Since $B\in \bigcap_{\xi\in\{-1\}\cup\beta}\Pi^1_\xi(\kappa)^+$, there is an $\alpha\in B$ such that $(\alpha,\in,A\cap \alpha)\models\varphi_\zeta[X\cap \alpha]$. Since $B$ is homogeneous for $\vec{X}$ we see that $X\cap \alpha= S_\alpha$ and thus $(\alpha,\in,A\cap \alpha)\models\varphi_\zeta[S_\alpha]$, a contradiction.

When $\beta=\eta+1$ is a successor ordinal the argument is very similar to Baumgartner's argument for \cite[Lemma 7.1]{MR0384553}. We must show that if $Q\subseteq\Pi^1_\eta(\kappa)^+$ then $S$ is a $\Pi^1_{\eta+2}$-indescribable subset of $\kappa$. Suppose $(\kappa,\in,A)\models\forall X\exists Y\psi(X,Y)$ where $\psi(X,Y)$ is a $\Pi^1_\eta$ formula. Further suppose that for each $\alpha\in S$ there is an $S_\alpha\subseteq \alpha$ such that $(\alpha,\in,A\cap \alpha,S_\alpha)\models\forall Y\lnot\psi(S_\alpha,Y)$. This defines a $(1,S)$-sequence $\vec{S}=\<S_\alpha\st\alpha\in S\>$. By assumption there is a $B\in Q$ which is homogeneous for $\vec{S}$. Let $X=\bigcup_{\alpha\in B}S_\alpha$. Then there is a $Y\subseteq \kappa$ such that $(\kappa,\in,A,X,Y)\models\psi(X,Y)$. Since $B$ is $\Pi^1_\eta$-indescribable, there is an $\alpha\in B$ such that $(\alpha,\in,A\cap \alpha,X\cap \alpha,Y\cap \alpha)\models\psi(X\cap \alpha,Y\cap \alpha)$. Since $X\cap \alpha=X_\alpha$, this is a contradiction. 
\end{proof}




\section{Basic properties of the ideals $\R^\alpha(\Pi^1_\beta(\kappa))$}\label{section_basic_properties}

In this section we begin our study of the ideals $\R^\alpha(\Pi^1_\beta(\kappa))$ obtained from iterating Feng's Ramsey operator on Bagaria's $\Pi^1_\beta$-indescribability ideals. The following straightforward lemmas will be used in Section \ref{section_indescribability_in_finite_ramseyness} and Section \ref{section_indescribability_in_infinite_ramseyness} below to prove that a proper containment holds between two particular ideals. 

Recall that if $S_\xi$ is a stationary subset of $\xi$ for all $\xi$ in some set $S\subseteq\kappa$ which is stationary in $\kappa$, then $\bigcup_{\xi\in S}S_\xi$ is stationary in $\kappa$. The next lemma shows that the analogous fact is true for the ideals $\R^\alpha(\Pi^1_\beta(\kappa))$.

\begin{lemma}\label{lemma_pos_union_of_pos_sets_is_pos}
Suppose $\alpha<\kappa$ and $\beta\in\{-1\}\cup\kappa$. Further suppose $S\in\R^\alpha(\Pi^1_\beta(\kappa))^+$ and for each $\xi\in S$ let $S_\xi\in\R^\alpha(\Pi^1_\beta(\xi))^+$. Then $\bigcup_{\xi\in S} S_\xi\in\R^\alpha(\Pi^1_\beta(\kappa))^+$.
\end{lemma}

\begin{proof}

Suppose $\alpha=0$. Fix $A\subseteq V_\kappa$ and let $\varphi$ be a $\Pi^1_\beta(\kappa)$ sentence such that $(V_\kappa,\in,A)\models\varphi$. Since $S\in\Pi^1_\beta(\kappa)^+$, there is a $\xi\in S$ such that $(V_\xi,\in,A\cap V_\xi)\models\varphi$. Now since $S_\xi\in\Pi^1_\beta(\xi)^+$, there is a $\zeta\in S_\xi$ such that $(V_\zeta,\in,A\cap V_\zeta)\models \varphi$. Hence $\bigcup_{\xi\in S}S_\xi\in\Pi^1_\beta(\kappa)^+$.

If $\alpha$ is a limit and the result holds for all ordinals less than $\alpha$, it can easily be checked that the result holds for $\alpha$ using the fact that $\R^\alpha(\Pi^1_\beta(\kappa))=\bigcup_{\zeta<\alpha}\R^\zeta(\Pi^1_\beta(\kappa))$.

Now suppose $\alpha>0$ is a successor ordinal and the result holds for $\alpha-1$, let us show that it holds for $\alpha$. Fix a regressive function $f:\left[\bigcup_{\xi\in S}S_\xi\right]^{<\omega}\to \kappa$. For each $\xi\in S$ there is a set $H_\xi\subseteq S_\xi$ in $\R^{\alpha-1}(\Pi^1_\beta(\xi))^+$ homogeneous for $f\restrict[S_\xi]^{<\omega}$. Since $S\in\R^\alpha(\Pi^1_\beta(\kappa))^+$, it follows by Corollary \ref{corollary_one_S} that the $(1,S)$-sequence $\vec{H}=\<H_\xi\st\xi\in S\>$ has a homogeneous set $H\subseteq S$ in $\R^{\alpha-1}(\Pi^1_\beta(\kappa))^+$. By our inductive hypothesis, $\bigcup_{\xi\in H} H_\xi\in\R^{\alpha-1}(\Pi^1_\beta(\kappa))^+$. It will suffice to show that $\bigcup_{\xi\in H} H_\xi$ is homogeneous for $f$. Suppose $\vec{\alpha},\vec{\beta}\in\left[\bigcup_{\xi\in H} H_\xi\right]^n$. By the homogeneity of $H$, it follows that there is a $\xi\in H$ such that $\vec{\alpha},\vec{\beta}\in [H_\xi]^n$. Since $H_\xi$ is homogeneous for $f\restrict[S_\xi]^{<\omega}$, we have $f(\vec{\alpha})=f(\vec{\beta})$.
\end{proof}

Recall that if $\kappa$ is a weakly compact cardinal, then the set of non--weakly compact cardinals less than $\kappa$ is a weakly compact subset of $\kappa$. The next lemma shows that the corresponding fact is true for the ideals $\R^\alpha(\Pi^1_\beta(\kappa))$.

\begin{lemma}\label{lemma_set_of_nons_is_positive}
For all $\alpha<\kappa$ and all $\beta\in\{-1\}\cup\kappa$, if $\kappa\in\R^\alpha(\Pi^1_\beta(\kappa))^+$ then the set
\[S=\{\xi<\kappa\st \xi\in\R^\alpha(\Pi^1_\beta(\xi))\}\]
is in $\R^\alpha(\Pi^1_\beta(\kappa))^+$.
\end{lemma}

\begin{proof}
Let $\kappa$ be the least counterexample. In other words, $\kappa$ is the least cardinal such that $\kappa\in\R^\alpha(\Pi^1_\beta(\kappa))^+$ and $S=\{\xi<\kappa\st\xi\in\R^\alpha(\Pi^1_\beta(\xi))\}\in\R^\alpha(\Pi^1_\beta(\kappa))$. Then the set $\kappa\setminus S$ is in $\R^\alpha(\Pi^1_\beta(\kappa))^*$ and hence also in $\R^\alpha(\Pi^1_\beta(\kappa))^+$. For each $\zeta\in\kappa\setminus S$, by the minimality of $\kappa$, the set $S_\zeta=S\cap\zeta$ is in $\R^\alpha(\Pi^1_\beta(\zeta))^+$. Thus, by Lemma \ref{lemma_pos_union_of_pos_sets_is_pos}, the set $S=\bigcup_{\zeta\in\kappa\setminus S}S_\zeta$ is in $\R^\alpha(\Pi^1_\beta(\kappa))^+$, a contradiction.
\end{proof}

\section{A first reflection result}\label{section_a_first_reflection_result}

Baumgartner showed \cite[Theorem 4.1]{MR0384553} that if $\kappa$ is a subtle cardinal then the set 
\[\{\xi<\kappa\st(\forall n<\omega)\ \xi\in\Pi^1_n(\xi)^+\}\]
is in the subtle filter. Since Ramsey cardinals are subtle, Baumgartner'€™s result shows that the existence of a Ramsey cardinal is strictly stronger than the existence of a cardinal that is $\Pi^1_n$-indescribable for every $n<\omega$. Our next goal will be to show that the existence of a Ramsey cardinal is strictly stronger than the existence of a cardinal $\kappa$ which is $\Pi^1_\beta$-indescribable for all $\beta<\kappa$; the proof is implicit in \cite{MR3894041} and is obtained by combining the methods of \cite{MR3894041}, \cite{MR2830415} and \cite[Theorem 17.33 and Exercise 17.29]{MR1940513}. In order to prove this result we will use the elementary embedding characterization of Ramseyness due to Mitchell \cite{MR534574} (see Theorem \ref{theorem_gitman} above) and further explored by Gitman \cite{MR2830415} and Sharpe-Welch \cite{MR2817562}. 

\begin{theorem}\label{theorem_ramsey_reflection}
If $S\in\R([\kappa]^{<\kappa})^+$, then the set
\[T=\{\xi<\kappa\st(\forall \beta<\xi)\ S\cap\xi\in\Pi^1_\beta(\xi)^+\}\]
is in $\R([\kappa]^{<\kappa})^*$.\marginpar{\tiny Need to generalize this so that it talks about reflection points of sets. This is needed for the following result.}
\end{theorem}

\begin{proof}\footnote{The author would like to thank Victoria Gitman for suggesting this proof.}
Suppose $S$ is Ramsey. To show that $T\in\R([\kappa]^{<\kappa})^*$ we must show that there is a set $A\subseteq\kappa$ such that whenever $M$ is a weak $\kappa$-model with $A,T\in M$ and whenever $j:M\to N$ is an elementary embedding satisfying properties (1)--(4) from Theorem \ref{theorem_gitman}, then it must be the case that $\kappa\in j(T)$. Take $A=S$. Since $S$ is Ramsey, by Theorem \ref{theorem_gitman}, we may let $M$ be a weak $\kappa$-model with $S,T\in M$ and suppose $j:M\to N$ is an elementary embedding satisfying properties (1)--(4) from Theorem \ref{theorem_gitman} such that $\kappa\in j(S)$. To show that $\kappa\in j(T)$ we must show that for every $\beta<\kappa$ we have $N\models$ $S\in\Pi^1_\beta(\kappa)^+$. Suppose not, that is, suppose that for some fixed $\beta<\kappa$, $N$ thinks $S$ is not a $\Pi^1_\beta$-indescribable subset of $\kappa$. Since $N$ thinks $\kappa$ is strongly inaccessible, it follows that $N$ thinks $S$ is not \emph{weakly} $\Pi^1_\beta$-indescribable.\marginpar{\tiny Check that when $\kappa$ is strongly inaccessible weak $\Pi^1_\beta$-indescribability implies $\Pi^1_\beta$-indescribability. We don'€™t necessarily need this, because we could code subsets of $V_\kappa$ as subsets of $\kappa$, but this seems more convenient.} Thus, there is an $R\in P(\kappa)^N$ and a $\Pi^1_\beta$ sentence $\varphi$ such that 

\[N\models\text{``$(\kappa,\in,R)\models\varphi$''}\]
and 
\[N\models \text{``$(\forall\xi\in S) (\xi,\in,R\cap\xi)\models\lnot\varphi$''}.\]
Now, for each $\xi <\kappa$ we have $R\cap\xi\in M$ because $P(\kappa)^M=P(\kappa)^N$. Furthermore, since $j$ is elementary and $\crit(j)=\kappa$, it follows that for each $\xi\in S$ we have $M\models$ ``$(\xi,\in,R\cap\xi)\models\lnot\varphi$.''€™ Since $S,R\in M$ we see that
\[M\models\text{``$(\forall\xi\in S) (\xi,\in,R\cap\xi)\models\lnot\varphi$''}.\]
By elementarity
\[N\models\text{``$(\forall\xi\in j(S)) (\xi,\in,j(R)\cap\xi)\models\lnot\varphi$''€™},\]
but this is a contradiction since $\kappa\in j(S)$.
\end{proof}

Next we show that Theorem \ref{theorem_ramsey_reflection} can, in a sense, be pushed up the Ramsey hierarchy.

\begin{theorem}\label{theorem_r_reflects_pi1n}
For all ordinals $\alpha<\kappa$, if $S\in\R^{\alpha+1}([\kappa]^{<\kappa})^+$ then the set
\[T=\{\xi<\kappa\st (\forall\beta<\xi)\  S\cap\xi\in\R^\alpha(\Pi^1_\beta(\xi))^+\}\]
is in $\R^{\alpha+1}([\kappa]^{<\kappa})^*$.\marginpar{\tiny The $\Pi^1_\beta(\xi)$ is redundant when $\alpha=\omega$ and $\beta<\omega$ because $\R^\omega([\kappa]^{<\kappa})=\R^\omega(\Pi^1_n(\kappa))$ for all $n<\omega$ by the ideal diagram.}
\end{theorem}

\begin{proof}
Let us proceed by induction on $\alpha$. If $\alpha=0$, the result follows directly from Theorem \ref{theorem_ramsey_reflection}.

Suppose $\alpha=\alpha_0+1$ is a successor ordinal and the result holds for ordinals less than $\alpha$. Let us show it holds for $\alpha$. Suppose not. Then $S\in\R^{\alpha_0+2}([\kappa]^{<\kappa})^+$ and the set
\[\kappa\setminus T=\{\xi<\kappa\st(\exists\beta<\xi)\ S\cap\xi\in\R^{\alpha_0+1}(\Pi^1_\beta(\xi))\}\]
is in $\R^{\alpha_0+2}([\kappa]^{<\kappa})^+$. Since $\R^{\alpha_0+2}([\kappa]^{<\kappa})$ is a normal ideal on $\kappa$, there is a fixed $\beta_0<\kappa$ such that the set
\[E=\{\xi<\kappa\st S\cap\xi\in\R^{\alpha_0+1}(\Pi^1_{\beta_0}(\xi))\land \xi>\beta_0\}\subseteq\kappa\setminus T\]
is in $\R^{\alpha_0+2}([\kappa]^{<\kappa})^+$. We will define a $(1,E)$-sequence $\vec{S}=\<E_\xi\st\xi\in E\>$. Without loss of generality, by intersecting with a club, we can assume that every element of $E$ is closed under G\"odel pairing. For each $\xi\in E$, let $f_\xi:[S\cap\xi]^{<\omega}\to \xi$ be a regressive function such that no homogeneous set for $f_\xi$ is in $\R^{\alpha_0}(\Pi^1_{\beta_0}(\xi))^+$. Let $E_\xi$ code $f_\xi$ as a subset of $\xi$ in a sufficiently nice way. Since $E\in\R^{\alpha_0+2}([\kappa]^{<\kappa})^+$, there is a set $X\in P(E)\cap\R^{\alpha_0+1}([\kappa]^{<\kappa})^+$ homogeneous for $\vec{E}$. Let $F=\bigcup_{\xi\in X}f_\xi$ and notice that $F:[S]^{<\omega}\to \kappa$ is a regressive function and $F\restrict [S\cap\xi]^{<\omega}=f_\xi$ for all $\xi\in X$. Since $S\in\R^{\alpha_0+2}([\kappa]^{<\kappa})^+$ there is an $H\in\R^{\alpha_0+1}([\kappa]^{<\kappa})^+$ homogeneous for $F$. By our inductive assumption, the set
\[C=\{\xi<\kappa\st H\cap\xi\in\R^{\alpha_0}(\Pi^1_{\beta_0}(\xi))^+\}\]
is in $\R^{\alpha_0+1}([\kappa]^{<\kappa})^*$ and since $X\in\R^{\alpha_0+1}([\kappa]^{<\kappa})^+$ it follows that $X\cap C\in\R^{\alpha_0+1}([\kappa]^{<\kappa})^+$. Choose an ordinal $\xi\in X\cap C$. Then $H\cap\xi\in\R^{\alpha_0}(\Pi^1_{\beta_0}(\xi))^+$ and since $H$ is homogeneous for $F$ and $\xi\in X$ we see that $H\cap\xi$ is homogeneous for $f_\xi$. This contradicts the fact that $f_\xi:[S\cap\xi]^{<\omega}\to \xi$ has no homogeneous set in $\R^{\alpha_0}(\Pi^1_{\beta_0}(\xi))^+$. 

Suppose $\alpha$ is a limit ordinal, the result holds for ordinals less than $\alpha$ and, for the sake of contradiction, the result is false for $\alpha$. Suppose $S\in\R^{\alpha+1}([\kappa]^{<\kappa})^+$ and the set
\[\kappa\setminus T=\{\xi<\kappa\st(\exists\beta<\kappa)\ S\cap\xi\in\R^\alpha(\Pi^1_\beta(\xi))\}\]
is in $\R^{\alpha+1}([\kappa]^{<\kappa})^+$. Since $\R^{\alpha+1}([\kappa]^{<\kappa})$ is a normal ideal and $\alpha<\kappa$ is a limit ordinal, there are fixed $\alpha_0<\alpha$ and $\beta_0<\kappa$ such that the set 
\[E=\{\xi<\kappa\st S\cap\xi\in\R^{\alpha_0+1}(\Pi^1_{\beta_0}(\xi))\}\]
is in $\R^{\alpha+1}([\kappa]^{<\kappa})^+$. The rest of the argument is essentially the same as that of the successor case. 
\end{proof}


Since we will refer to it later, let us state the following corollary which asserts that ``$\exists\kappa$ $\kappa\in\R^{\alpha+1}([\kappa]^{<\kappa})^+$'' is strictly stronger than ``$\exists\kappa$ ($\forall\beta<\kappa$) $\kappa\in\R^\alpha(\Pi^1_\beta(\kappa))^+$''.

\begin{corollary}\label{corollary_ramseyness_reflects_indescribability}
For all ordinals $\alpha<\kappa$, if $\kappa\in\R^{\alpha+1}([\kappa]^{<\kappa})^+$ then the set
\[\{\xi<\kappa\st(\forall\beta<\xi)\ \xi\in\R^\alpha(\Pi^1_\beta(\xi))^+\}\]
is in $\R^{\alpha+1}([\kappa]^{<\kappa})^*$.
\end{corollary}

\section{Describing degrees of Ramseyness}\label{section_describing_ramseyness}


In order to prove that certain relationships hold between ideals of the form $\R^\alpha(\Pi^1_\beta(\kappa))$, we will need to know what $\xi$ will suffice to be able to express the fact that a set $S\subseteq\kappa$ is in $\R^\alpha(\Pi^1_\beta(\kappa))^+$ using a $\Pi^1_\xi$ sentence. The following lemma is a generalization of a result of Sharpe and Welch \cite[Remark 3.17]{MR2817562}, which states that ``$X\in\R^\alpha([\kappa]^{<\kappa})^+$'' is a $\Pi^1_{2\cdot(1+\alpha)}$ property.

\begin{lemma}[The case $\beta=-1$ is due to Sharpe and Welch]\label{lemma_complexity}
Suppose $\beta$ is an ordinal. For each ordinal $\alpha$, if $\alpha>\omega$ let $\alpha=\bar{\alpha}+m_\alpha$ where $\bar{\alpha}$ is the greatest limit ordinal which is less or equal to $\alpha$ and $m_\alpha<\omega$. Define an ordinal $\gamma(\alpha,\beta)$ by
\[
\gamma(\alpha,\beta)=\begin{cases}
\beta+2\alpha+1 & \text{if $\alpha<\omega$}\\
\beta+\alpha & \text{if $\alpha$ is a limit}\\
\beta+\bar{\alpha}+2m_\alpha & \text{if $\alpha>\omega$ is a successor}\\
\end{cases}
\]
Then, for all ordinals $\alpha$, there is a $\Pi^1_{\gamma(\alpha,\beta)}$ sentence $\varphi$ such that for all cardinals $\delta$ with $\max(\alpha,\beta)<\delta$ and all sets $X\subseteq\delta$ we have
\[X\in\R^\alpha(\Pi^1_\beta(\delta))^+\iff (V_\delta,\in,X)\models\varphi.\]
\end{lemma}

\begin{proof}
First we consider the case in which $\alpha<\omega$. If $\alpha=0$ then the result holds because there is a $\Pi^1_{\beta+1}$ sentence $\varphi$ such that if $\beta<\delta$ and $X\subseteq\delta$ then $X\in\R^0(\Pi^1_\beta(\delta))^+=\Pi^1_\beta(\delta)^+$ if and only if $(V_\delta,\in,X)\models\varphi$ (see \cite{MR3894041}). Assume $0<\alpha<\omega$ and the result holds for the ordinal $\alpha-1<\omega$. Then there is a $\Pi^1_{\beta+2\alpha-1}$ sentence $\psi$ such that whenever $\max(\alpha,\beta)<\delta$ and $X\subseteq\delta$ we have $X\in\R^{\alpha-1}(\Pi^1_\beta)^+$ if and only if $(V_\delta,\in,X)\models\psi$. By definition of the Ramsey operator, for any relevant $\delta$ and $X\subseteq\delta$, we have $X\in\R^\alpha(\Pi^1_\beta(\delta))^+$ if and only if for every regressive function $f:[X]^{<\omega}\to \delta$ there is a set $H\in\R^{\alpha-1}(\Pi^1_\beta(\delta))^+$ homogeneous for $f$. Thus there is a $\Pi^1_{\beta+2\alpha+1}$ sentence $\varphi$ (namely, the sentence ``$\forall f\exists H \psi$'') such that $X\in\R^\alpha(\Pi^1_\beta(\delta))^+$ if and only if $(V_\delta,\in,X)\models\varphi$.

Suppose $\alpha<\kappa$ is a limit ordinal and the result holds for all ordinals $\eta€™<\alpha$. By definition of the Ramsey hierarchy, for all relevant $\delta$ we have $\R^\alpha(\Pi^1_\beta(\delta))=\bigcup_{\xi<\alpha}\R^\xi(\Pi^1_\beta(\delta))$, and thus, for sets $X\subseteq\delta$ we have 
\[X\in\R^\alpha(\Pi^1_\beta(\delta))^+\iff V_\delta\models\bigwedge_{\xi<\alpha}\left(X\in\R^\xi(\Pi^1_\beta(\delta))^+\right).\]
For each $\xi<\alpha$ there is a $\Pi^1_{\gamma(\xi,\beta)}$ sentence $\varphi_\xi$ such that $X\in\R^\xi(\Pi^1_\beta(\delta))^+$ if and only if $(V_\delta,\in,X)\models\varphi_\xi$. Since the sequence $\<\gamma(\xi,\beta)\st\xi<\alpha\>$ is strictly increasing and $\gamma(\xi,\beta)<\beta+\alpha$ for all $\xi<\alpha$, it follows that there is a $\Pi^1_{\beta+\alpha}$ sentence $\varphi$ such that $X\in\R^\alpha(\Pi^1_\beta(\delta))^+$ if and only if $(V_\delta,\in,X)\models\varphi$ for all relevant $\delta$ and $X$.

Suppose $\alpha=\bar{\alpha}+m_\alpha>\omega$ is a successor ordinal and the result holds for all ordinals less than $\alpha$. Notice that $m_\alpha\geq 1$ since $\bar{\alpha}$ is a limit ordinal. Suppose $m_\alpha=1$. By our inductive hypothesis there is a $\Pi^1_{\beta+\bar{\alpha}}$ sentence $\psi$ such that for all relevant $\delta$ and $X\subseteq\delta$ we have $X\in\R^{\bar{\alpha}}(\Pi^1_\beta(\delta))^+$ if and only if $(V_\delta,\in,X)\models\psi$. By definition of the Ramsey operator, for all relevant $\delta$ and $X\subseteq\delta$ we have $X\in\R^{\bar{\alpha}+1}(\Pi^1_\beta(\delta))^+$ if and only if for every regressive function $f:[X]^{<\omega}\to \kappa$ there is a set $H\in\R^{\bar{\alpha}}(\Pi^1_\beta(\delta))^+$ homogeneous for $f$. This implies that there is a $\Pi^1_{\beta+\bar{\alpha}+2}$ sentence $\varphi$ such that for all relevant $\delta$ and $X\subseteq\delta$ we have $X\in\R^{\bar{\alpha}+1}(\Pi^1_\beta(\delta))^+$ if and only if $(V_\delta,\in,X)\models\varphi$. Now, suppose $m_\alpha>1$ and assume the result holds for the ordinal $\bar{\alpha}+m_\alpha-1$. Then for all relevant $\delta$ and $X\subseteq\delta$ there is a $\Pi^1_{\beta+\bar{\alpha}+2(m_\alpha-1)}$ sentence $\psi$ such that $X\in\R^{\bar{\alpha}+m_\alpha-1}(\Pi^1_\beta(\delta))^+$ if and only if $(V_\delta,\in,X)\models\psi$. For all relevant $\delta$ and $X\subseteq\delta$ we have $X\in\R^{\bar{\alpha}+m_\alpha}(\Pi^1_\beta(\delta))^+$ if and only if for every regressive function $f:[X]^{<\omega}\to \kappa$ there is a set $H\in\R^{\bar{\alpha}+m_\alpha-1}(\Pi^1_\beta(\delta))^+$ homogeneous for $f$. Since $H\in\R^{\bar{\alpha}+m_\alpha-1}(\Pi^1_\beta(\delta))^+$ is expressible by a $\Pi^1_{\beta+\bar{\alpha}+2(m_\alpha-1)}$ sentence $\psi$ over $(V_\delta,\in,H)$, it follows that there is a $\Pi^1_{\beta+\bar{\alpha}+2m_\alpha}$ sentence $\varphi$ such that $X\in\R^{\bar{\alpha}+m_\alpha}(\Pi^1_\beta(\delta))^+$ if and only if $(V_\delta,\in,X)\models\varphi$.
\end{proof}

\section{Indescribability in finite degrees of Ramseyness}\label{section_indescribability_in_finite_ramseyness}

Next we prove that for $0<m<\omega$ and $\beta<\kappa$, the ideal $\R^m(\Pi^1_\beta(\kappa))$ is obtained by using a generating set consisting of the pre-Ramsey operator applied to the ideal $\R^{m-1}(\Pi^1_\beta(\kappa))$ one-level down in the Ramsey hierarchy together with the ideal $\R^{m-1}(\Pi^1_{\beta+2}(\kappa))$. This result also gives more information about the ideals $\R^m([\kappa]^{<\kappa})$ considered by Feng \cite[Theorem 4.8]{MR1077260}.

\begin{theorem}\label{theorem_finite_ideal_diagram}
For all $0<m<\omega$ and all $\beta\in\{-1\}\cup\kappa$, if $\kappa\in\R^m(\Pi^1_\beta(\kappa))^+$ then
\[\R^m(\Pi^1_\beta(\kappa))=\overline{\R_0(\R^{m-1}(\Pi^1_\beta(\kappa)))\cup\R^{m-1}(\Pi^1_{\beta+2}(\kappa))}.\]
\end{theorem}

\begin{proof}
We proceed by induction on $m$. For the base case of the induction in which $m=1$, we will show that for all $\beta\in\{-1\}\cup\kappa$ we have
\[\R(\Pi^1_\beta(\kappa))=\overline{\R_0(\Pi^1_\beta(\kappa))\cup\Pi^1_{\beta+2}(\kappa)}.\]
Fix $\beta\in\{-1\}\cup\kappa$ and let $I=\overline{\R_0(\Pi^1_\beta(\kappa))\cup\Pi^1_{\beta+2}(\kappa)}$. We will show that $X\in I^+$ if and only if $X\in\R(\Pi^1_\beta(\kappa))^+$.

Suppose $X\in I^+$ and $X\in\R(\Pi^1_\beta(\kappa))$. Let $f:[X]^{<\omega}\to \kappa$ be a regressive function and suppose that every homogeneous set for $f$ is not $\Pi^1_\beta$-indescribable. This can be expressed by a $\Pi^1_{\beta+2}$ sentence $\varphi$ over $(V_\kappa,\in,X,f)$, and thus the set
\[C=\{\xi<\kappa\st (V_\xi,\in,X\cap V_\xi,f\cap V_\xi)\models\varphi\}\] 
is in $\Pi^1_{\beta+2}(\kappa)^*$. Since $X\notin I$, $X$ is not the union of a set in $\R_0(\Pi^1_\beta(\kappa))$ and a set in $\Pi^1_{\beta+2}(\kappa)$, and since $X=(X\cap C)\cup(X\setminus C)$, it follows that $X\cap C\notin \R_0(\Pi^1_\beta(\kappa))$. Thus, by definition of $\R_0(\Pi^1_\beta(\kappa))$, there is a $\xi\in X\cap C$ with $\xi>\beta$ such that there is a set $H\subseteq X\cap C\cap\xi$ which is $\Pi^1_\beta$-indescribable in $\xi$ and homogeneous for $f$. This contradicts $\xi\in C$.

Now suppose $X\in\R(\Pi^1_\beta(\kappa))^+$. By Remark \ref{remark_ideal_generated}, it suffices to show that $X\in\R_0(\Pi^1_\beta(\kappa))^+$ and $X\in\Pi^1_{\beta+2}(\kappa)^+$. To see that $X\in\Pi^1_{\beta+2}(\kappa)^+$ notice that every $(1,X)$-sequence $\vec{X}=\<X_\xi\st\xi\in X\>$ has a homogeneous set in $\Pi^1_\beta(\kappa)^+$, and thus by Lemma \ref{lemma_baumgartner_bagaria}, $X$ is $\Pi^1_{\beta+2}$-indescribable. Let us show $X\in\R_0(\Pi^1_\beta(\kappa))^+$. Fix a regressive function $f:[X]^{<\omega}\to \kappa$ and a club $C\subseteq\kappa$. Since $\R(\Pi^1_\beta(\kappa))$ is a normal ideal it follows that $X\cap C\in\R(\Pi^1_\beta(\kappa))^+$ and thus there is a set $H\subseteq X\cap C$ which is $\Pi^1_\beta$-indescribable in $\kappa$ and homogeneous for $f$. The fact that $H$ is $\Pi^1_\beta$-indescribable can be expressed by a $\Pi^1_{\beta+1}$ sentence $\varphi$ over $(V_\kappa,\in,H)$. Since $X\cap C\in\Pi^1_{\beta+2}(\kappa)^+$, it follows that there is a $\xi\in X\cap C$ with $\xi>\beta$ such that $(V_\xi,\in,H\cap\xi)\models \varphi$ and hence $H\cap\xi\subseteq X\cap C\cap\xi$ is $\Pi^1_\beta$-indescribable in $\xi$ and homogeneous for $f$. Thus $X\in\R_0(\Pi^1_\beta(\kappa))^+$.

For the inductive step, suppose that for all $k<m$ and all $\beta\in\{-1\}\cup\kappa$ we have
\[\R^k(\Pi^1_\beta(\kappa))=\overline{\R_0(\R^{k-1}(\Pi^1_\beta(\kappa)))\cup\R^{k-1}(\Pi^1_{\beta+2}(\kappa))}.\]
Fix $\beta\in\{-1\}\cup\kappa$. Let us argue that 
\[\R^m(\Pi^1_\beta(\kappa))=\overline{\R_0(\R^{m-1}(\Pi^1_\beta(\kappa)))\cup\R^{m-1}(\Pi^1_{\beta+2}(\kappa))}.\] Let $I=\overline{\R_0(\R^{m-1}(\Pi^1_\beta(\kappa)))\cup\R^{m-1}(\Pi^1_{\beta+2}(\kappa))}$. We will show that $X\in I^+$ if and only if $X\in\R^m(\Pi^1_\beta(\kappa))^+$.

Suppose $X\in I^+$. For the sake of contradiction suppose that $X\in\R^m(\Pi^1_\beta(\kappa))$. Then there is a regressive function $f:[X]^{<\omega}\to \kappa$ such that every homogeneous set for $f$ is in $\R^{m-1}(\Pi^1_\beta(\kappa))$. By Lemma \ref{lemma_complexity}, the fact that every homogeneous set for $f$ is in $\R^{m-1}(\Pi^1_\beta(\kappa))$ can be expressed by a $\Pi^1_{\beta+2m}$ sentence $\varphi$ over $(V_\kappa,\in,X,f)$. Thus the set $C=\{\alpha<\kappa\st (V_\alpha,\in,X\cap\alpha,f\cap V_\alpha)\models\varphi\}$ is in $\Pi^1_{\beta+2m}(\kappa)^*$. Let us show that our inductive assumption implies that $\Pi^1_{\beta+2m}(\kappa)\subseteq\R^{m-1}(\Pi^1_{\beta+2}(\kappa))$. From our inductive assumption, it follows that
\begin{align*}
\R^{m-1}(\Pi^1_{\beta+2}(\kappa))&=\overline{\R_0(\R^{m-2}(\Pi^1_{\beta+2}(\kappa)))\cup\R^{m-2}(\Pi^1_{\beta+4}(\kappa))}\\
	&=\overline{\R_0(\R^{m-2}(\Pi^1_{\beta+2}(\kappa)))\cup\R_0(\R^{m-3}(\Pi^1_{\beta+4}(\kappa)))\cup\R^{m-3}(\Pi^1_{\beta+6}(\kappa))}\\
	&\ \ \vdots\\
	&=\overline{\left(\bigcup_{i=1}^{m-1} \R_0(\R^{m-(i+1)}(\Pi^1_{\beta+2i}(\kappa)))\right)\cup \Pi^1_{\beta+2m}(\kappa)}.
\end{align*}
Thus $C\in\Pi^1_{\beta+2m}(\kappa)^*\subseteq\R^{m-1}(\Pi^1_{\beta+2}(\kappa))^*$. Since $X\in I^+$, $X$ is not the union of a set in $\R_0(\R^{m-1}(\Pi^1_\beta(\kappa)))$ and a set in $\R^{m-1}(\Pi^1_{\beta+2}(\kappa))$, and since $X=(X\cap C)\cup (X\setminus C)$, it follows that $X\cap C\notin \R_0(\R^{m-1}(\Pi^1_\beta(\kappa)))$. Thus, by definition of $\R_0(\R^{m-1}(\Pi^1_\beta(\kappa)))$, there is a $\xi\in X\cap C$ with $\xi>\beta$ and an $H\subseteq X\cap C\cap\xi$ in $\R^{m-1}(\Pi^1_\beta(\xi))^+$ homogeneous for $f$. But, this contradicts $\xi\in C$.

Suppose $X\in \R^m(\Pi^1_\beta(\kappa))^+$. By Remark \ref{remark_ideal_generated}, it suffices to show that $X\in\R^{m-1}(\Pi^1_{\beta+2}(\kappa))^+$ and $X\in\R_0(\R^{m-1}(\Pi^1_\beta(\kappa))^+$. Since $X\in \R^m(\Pi^1_\beta(\kappa))^+$, every regressive function $f:[X]^{<\omega}\to \kappa$ has a homogeneous set $H\in\R^{m-1}(\Pi^1_\beta(\kappa))^+$. From our inductive assumption it follows that 
\[\R^{m-1}(\Pi^1_\beta(\kappa))=\overline{\R_0(\R^{m-2}(\Pi^1_\beta(\kappa))\cup\R^{m-2}(\Pi^1_{\beta+2}(\kappa))}\] 
and thus, every regressive function $f:[X]^{<\omega}\to \kappa$ has a homogeneous set $H\in\R^{m-2}(\Pi^1_{\beta+2}(\kappa))^+$. In other words, $X\in\R^{m-1}(\Pi^1_{\beta+2}(\kappa))^+$. It remains to show that $X\in\R_0(\R^{m-1}(\Pi^1_\beta(\kappa))^+$. Fix a regressive function $f:[X]^{<\omega}\to \kappa$ and a club $C\subseteq\kappa$. Since $\R^m(\Pi^1_\beta(\kappa))$ is a normal ideal it follows that $X\cap C\in\R^m(\Pi^1_\beta(\kappa))^+$ and thus every $(1,X\cap C)$-sequence has a homogeneous set in $\R^{m-1}(\Pi^1_\beta(\kappa))^+$ (by Theorem \ref{theorem_ramsey_equiv}). From our inductive assumption we see that every element of $\R^{m-1}(\Pi^1_\beta(\kappa))^+$ is $\Pi^1_{\beta+2m-2}$-indescribable, and thus, by Lemma \ref{lemma_baumgartner_bagaria}, $X\cap C$ is $\Pi^1_{\beta+2m}$-indescribable. Since $X\cap C\in\R^m(\Pi^1_\beta(\kappa))^+$ there is a set $H\subseteq X\cap C$ in $\R^{m-1}(\Pi^1_\beta(\kappa))^+$ which is homogeneous for $f$. By Lemma \ref{lemma_complexity}, the fact that $H\in\R^{m-1}(\Pi^1_\beta(\kappa))^+$ can be expressed by a $\Pi^1_{\beta+2m-1}$ sentence $\varphi$ over $(V_\kappa,\in,H)$. Since $X\cap C$ is $\Pi^1_{\beta+2m}$-indescribable we see that there is a $\xi\in X\cap C$ such that $(V_\xi,\in,H\cap\xi)\models\varphi$, in other words, $H\cap\xi\subseteq X\cap C\cap\xi$ and $H\cap \xi\in\R^{m-1}(\Pi^1_\beta(\xi))^+$. Thus $X\in\R_0(\R^{m-1}(\Pi^1_\beta(\kappa)))^+$.
\end{proof}

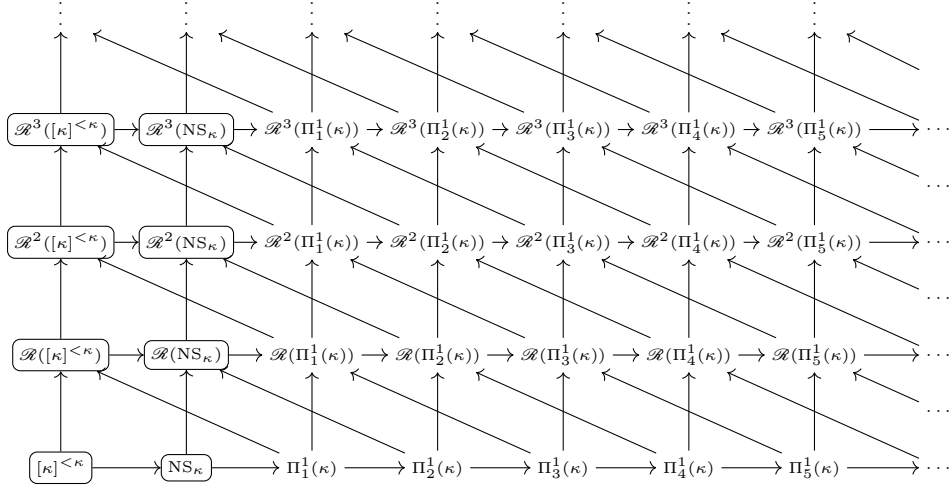
\begin{figure}
\centering

\begin{tikzpicture}[x = 1.67cm , y = 1.5cm]
\tiny
\node[draw,rounded corners=0.1cm] (pi1-1) at (0,0) {$[\kappa]^{<\kappa}$};
\node[draw,rounded corners=0.1cm] (pi10) at (1,0) {$\NS_\kappa$};
\node (pi11) at (2,0) {$\Pi^1_1(\kappa)$};
\node (pi12) at (3,0) {$\Pi^1_2(\kappa)$};
\node (pi13) at (4,0) {$\Pi^1_3(\kappa)$};
\node (pi14) at (5,0) {$\Pi^1_4(\kappa)$};
\node (pi15) at (6,0) {$\Pi^1_5(\kappa)$};
\node (dots) at (7,0) {$\cdots$};
\node (dots+) at (7,0.5) {$\cdots$};

\node[draw,rounded corners=0.1cm] (rpi1-1) at (0,1) {$\R([\kappa]^{<\kappa})$};
\node[draw,rounded corners=0.1cm] (rpi10) at (1,1) {$\R(\NS_\kappa)$};
\node (rpi11) at (2,1) {$\R(\Pi^1_1(\kappa))$};
\node (rpi12) at (3,1) {$\R(\Pi^1_2(\kappa))$};
\node (rpi13) at (4,1) {$\R(\Pi^1_3(\kappa))$};
\node (rpi14) at (5,1) {$\R(\Pi^1_4(\kappa))$};
\node (rpi15) at (6,1) {$\R(\Pi^1_5(\kappa))$};
\node (rdots) at (7,1) {$\cdots$};
\node (rdots+) at (7,1.5) {$\cdots$};

\node[draw,rounded corners=0.1cm] (r2pi1-1) at (0,2) {$\R^2([\kappa]^{<\kappa})$};
\node[draw,rounded corners=0.1cm] (r2pi10) at (1,2) {$\R^2(\NS_\kappa)$};
\node (r2pi11) at (2,2) {$\R^2(\Pi^1_1(\kappa))$};
\node (r2pi12) at (3,2) {$\R^2(\Pi^1_2(\kappa))$};
\node (r2pi13) at (4,2) {$\R^2(\Pi^1_3(\kappa))$};
\node (r2pi14) at (5,2) {$\R^2(\Pi^1_4(\kappa))$};
\node (r2pi15) at (6,2) {$\R^2(\Pi^1_5(\kappa))$};
\node (r2dots) at (7,2) {$\cdots$};
\node (r2dots+) at (7,2.5) {$\cdots$};

\node[draw,rounded corners=0.1cm] (r3pi1-1) at (0,3) {$\R^3([\kappa]^{<\kappa})$};
\node[draw,rounded corners=0.1cm] (r3pi10) at (1,3) {$\R^3(\NS_\kappa)$};
\node (r3pi11) at (2,3) {$\R^3(\Pi^1_1(\kappa))$};
\node (r3pi12) at (3,3) {$\R^3(\Pi^1_2(\kappa))$};
\node (r3pi13) at (4,3) {$\R^3(\Pi^1_3(\kappa))$};
\node (r3pi14) at (5,3) {$\R^3(\Pi^1_4(\kappa))$};
\node (r3pi15) at (6,3) {$\R^3(\Pi^1_5(\kappa))$};
\node (r3dots) at (7,3) {$\cdots$};
\node (r3dots+) at (6.89,3.5) {};

\node (r4pi1-1) at (0,4.1) {$\vdots$};
\node (r4pi1-1+) at (0.06,3.9) {};
\node (r4pi1-1++) at (0.2,3.88) {};

\node (r4pi10) at (1,4.1) {$\vdots$};
\node (r4pi10+) at (1.06,3.9) {};
\node (r4pi10++) at (1.2,3.88) {};

\node (r4pi11) at (2,4.1) {$\vdots$};
\node (r4pi11+) at (2.06,3.9) {};
\node (r4pi11++) at (2.2,3.88) {};

\node (r4pi12) at (3,4.1) {$\vdots$};
\node (r4pi12+) at (3.06,3.9) {};
\node (r4pi12++) at (3.2,3.88) {};

\node (r4pi13) at (4,4.1) {$\vdots$};
\node (r4pi13+) at (4.06,3.9) {};
\node (r4pi13++) at (4.2,3.88) {};

\node (r4pi14) at (5,4.1) {$\vdots$};
\node (r4pi14+) at (5.06,3.9) {};
\node (r4pi14++) at (5.2,3.88) {};

\node (r4pi15) at (6,4.1) {$\vdots$};
\node (r4pi15+) at (6.06,3.9) {};
\node (r4pi15++) at (6.2,3.88) {};


\draw [->] (pi1-1) -- (pi10);
\draw [->] (pi10) -- (pi11);
\draw [->] (pi11) -- (pi12);
\draw [->] (pi12) -- (pi13);
\draw [->] (pi13) -- (pi14);
\draw [->] (pi14) -- (pi15);
\draw [->] (pi15) -- (dots);

\draw [->] (pi1-1) -- (rpi1-1);
\draw [->] (pi11) -- (rpi1-1);

\draw [->] (pi10) -- (rpi10);
\draw [->] (pi12) -- (rpi10);

\draw [->] (pi11) -- (rpi11);
\draw [->] (pi13) -- (rpi11);

\draw [->] (pi12) -- (rpi12);
\draw [->] (pi14) -- (rpi12);

\draw [->] (pi13) -- (rpi13);
\draw [->] (pi15) -- (rpi13);

\draw [->] (pi14) -- (rpi14);
\draw [->] (dots) -- (rpi14);

\draw [->] (pi15) -- (rpi15);
\draw [->] (dots+) -- (rpi15);

\draw [->] (rpi1-1) -- (rpi10);
\draw [->] (rpi10) -- (rpi11);
\draw [->] (rpi11) -- (rpi12);
\draw [->] (rpi12) -- (rpi13);
\draw [->] (rpi13) -- (rpi14);
\draw [->] (rpi14) -- (rpi15);
\draw [->] (rpi15) -- (rdots);

\draw [->] (rpi1-1) -- (r2pi1-1);
\draw [->] (rpi11) -- (r2pi1-1);

\draw [->] (rpi10) -- (r2pi10);
\draw [->] (rpi12) -- (r2pi10);

\draw [->] (rpi11) -- (r2pi11);
\draw [->] (rpi13) -- (r2pi11);

\draw [->] (rpi12) -- (r2pi12);
\draw [->] (rpi14) -- (r2pi12);

\draw [->] (rpi13) -- (r2pi13);
\draw [->] (rpi15) -- (r2pi13);

\draw [->] (rpi14) -- (r2pi14);
\draw [->] (rdots) -- (r2pi14);

\draw [->] (rpi15) -- (r2pi15);
\draw [->] (rdots+) -- (r2pi15);

\draw [->] (r2pi1-1) -- (r2pi10);
\draw [->] (r2pi10) -- (r2pi11);
\draw [->] (r2pi11) -- (r2pi12);
\draw [->] (r2pi12) -- (r2pi13);
\draw [->] (r2pi13) -- (r2pi14);
\draw [->] (r2pi14) -- (r2pi15);
\draw [->] (r2pi15) -- (r2dots);

\draw [->] (r2pi1-1) -- (r3pi1-1);
\draw [->] (r2pi11) -- (r3pi1-1);

\draw [->] (r2pi10) -- (r3pi10);
\draw [->] (r2pi12) -- (r3pi10);

\draw [->] (r2pi11) -- (r3pi11);
\draw [->] (r2pi13) -- (r3pi11);

\draw [->] (r2pi12) -- (r3pi12);
\draw [->] (r2pi14) -- (r3pi12);

\draw [->] (r2pi13) -- (r3pi13);
\draw [->] (r2pi15) -- (r3pi13);

\draw [->] (r2pi14) -- (r3pi14);
\draw [->] (r2dots) -- (r3pi14);

\draw [->] (r2pi15) -- (r3pi15);
\draw [->] (r2dots+) -- (r3pi15);

\draw [->] (r3pi1-1) -- (r3pi10);
\draw [->] (r3pi10) -- (r3pi11);
\draw [->] (r3pi11) -- (r3pi12);
\draw [->] (r3pi12) -- (r3pi13);
\draw [->] (r3pi13) -- (r3pi14);
\draw [->] (r3pi14) -- (r3pi15);
\draw [->] (r3pi15) -- (r3dots);

\draw [->] (r3pi1-1) -- (r4pi1-1.south);
\draw [->] (r3pi11) -- (r4pi1-1++);

\draw [->] (r3pi10) -- (r4pi10.south);
\draw [->] (r3pi12) -- (r4pi10++);

\draw [->] (r3pi11) -- (r4pi11.south);
\draw [->] (r3pi13) -- (r4pi11++);

\draw [->] (r3pi12) -- (r4pi12.south);
\draw [->] (r3pi14) -- (r4pi12++);

\draw [->] (r3pi13) -- (r4pi13.south);
\draw [->] (r3pi15) -- (r4pi13++);

\draw [->] (r3pi14) -- (r4pi14.south);
\draw [->] (r3dots) -- (r4pi14++);

\draw [->] (r3pi15) -- (r4pi15.south);
\draw [->] (r3dots+) -- (r4pi15++);

\end{tikzpicture}
\caption{\tiny Diagram of ideal containments of $\R^m(\Pi^1_\beta(\kappa))$ for $m<\omega$ and $\beta<\kappa$. The circled ideals are those in Feng's original hierarchy. Each arrow $\rightarrow$ indicates a containment $\subseteq$ which is proper when the ideals are nontrivial by Theorem \ref{theorem_proper_containments_in_finite_diagram}.}\label{figure_finite_ideal_diagram}

\end{figure}

The next corollary generalizes a result of Feng \cite[Theorem 4.8]{MR1077260} and indicates precisely the degree of indescribability that can be derived from a given finite degree of Ramseyness.

\begin{corollary}\label{corollary_indescribability_in_finite_ramseyness}
For all $0<m<\omega$ and all $\beta\in\{-1\}\cup\kappa$, if $\kappa\in\R^m(\Pi^1_\beta(\kappa))^+$ then 
\[\R^m(\Pi^1_\beta(\kappa))=\overline{\R_0(\R^{m-1}(\Pi^1_\beta(\kappa)))\cup\Pi^1_{\beta+2m}(\kappa)}.\]\marginpar{\tiny We really only need $\R_0(\R^{m-1}(\Pi^1_n(\kappa)))$ here since the relevant ideals form a chain.}
\end{corollary}

\begin{proof}
Fix $\beta\in\{-1\}\cup\kappa$. If $m=1$ the result follows directly from Theorem \ref{theorem_finite_ideal_diagram}. Now suppose $1\leq m<\omega$ and the result holds for $m$, let us show that it holds for $m+1$. By Theorem \ref{theorem_finite_ideal_diagram}, we have
\[\R^{m+1}(\Pi^1_\beta(\kappa))=\overline{\R_0(\R^m(\Pi^1_\beta(\kappa)))\cup \R^m(\Pi^1_{\beta+2}(\kappa))}.\]
By our inductive assumption we see that
\[\R^{m+1}(\Pi^1_\beta(\kappa))=\overline{\R_0(\R^m(\Pi^1_\beta(\kappa)))\cup \R_0(\R^{m-1}(\Pi^1_{\beta+2}(\kappa)))\cup\Pi^1_{\beta+2m+2}(\kappa)}.\]
From Theorem \ref{theorem_finite_ideal_diagram} it follows that $\R_0(\R^{m-1}(\Pi^1_{\beta+2}(\kappa)))\subseteq \R_0(\R^m(\Pi^1_\beta(\kappa)))$ and hence
\[\R^{m+1}(\Pi^1_\beta(\kappa))=\overline{\R_0(\R^m(\Pi^1_\beta(\kappa)))\cup\Pi^1_{\beta+2(m+1)}(\kappa)}.\]
Thus the result holds for $m+1$.
\end{proof}

As in Baumgartner's characterization of ineffability \cite[Section 7]{MR0384553} in terms of the subtle ideal and the $\Pi^1_2$-indescribability ideal, and as in Baumgartner's characterization of Ramseyness \cite[Theorem 4.4 and Theorem 4.5]{MR0540770} in terms of the pre-Ramsey and $\Pi^1_1$-indescribability ideals, the next corollary demonstrates that large cardinal ideals are, in a sense, necessary for certain results.

\begin{corollary}\label{corollary_necessity}
For all $0<m<\omega$ and all $\beta\in\{-1\}\cup\kappa$ we have $\kappa\in\R^m(\Pi^1_\beta(\kappa))^+$ if and only if \emph{both} of the following properties hold.
\begin{enumerate}
\item $\kappa\in\R_0(\R^{m-1}(\Pi^1_\beta(\kappa)))^+$ and $\kappa\in\Pi^1_{\beta+2m}(\kappa)^+$.
\item The ideal generated by $\R_0(\R^{m-1}(\Pi^1_\beta(\kappa)))\cup\Pi^1_{\beta+2m}(\kappa)$ is nontrivial and equals $\R^m(\Pi^1_\beta(\kappa))$.
\end{enumerate}
Moreover, reference to the ideals in the above characterization cannot be removed because the least cardinal $\kappa$ such that $\kappa\in\R_0(\R^{m-1}(\Pi^1_\beta(\kappa)))^+$ and $\kappa\in\Pi^1_{\beta+2m}(\kappa)^+$ is not in $\R^m(\Pi^1_\beta(\kappa))^+$.
\end{corollary}

\begin{proof}
The characterization of $\kappa\in\R^m(\Pi^1_\beta(\kappa))^+$ follows directly from Corollary \ref{corollary_indescribability_in_finite_ramseyness}. For the additional statement, let us show that if $\kappa\in\R^m(\Pi^1_\beta(\kappa))^+$ then there are many cardinals $\xi<\kappa$ such that $\xi\in\R_0(\R^{m-1}(\Pi^1_\beta(\xi)))^+$ and $\xi\in\Pi^1_{\beta+2m}(\xi)^+$. Suppose $\kappa\in\R^m(\Pi^1_\beta(\kappa))^+$. Since 
\[\R^m(\Pi^1_\beta(\kappa))=\overline{\R_0(\R^{m-1}(\Pi^1_\beta(\kappa)))\cup\Pi^1_{\beta+2m}(\kappa)},\]
it follows that $\kappa\in\R_0(\R^{m-1}(\Pi^1_\beta(\kappa)))^+$ and $\kappa\in\Pi^1_{\beta+2m}(\kappa)^+$. Now $\kappa\in\R_0(\R^{m-1}(\Pi^1_\beta(\kappa)))^+$ is $\Pi^1_1$-expressible over $V_\kappa$ and thus the set
\[C_0=\{\xi<\kappa\st \xi\in\R_0(\R^{m-1}(\Pi^1_\beta(\xi)))^+\}\]
is in $\Pi^1_1(\kappa)^*\subseteq\Pi^1_{\beta+2m}(\kappa)^*\subseteq\R^m(\Pi^1_\beta(\kappa))^*$.
Furthermore, by Corollary \ref{corollary_ramseyness_reflects_indescribability}, we see that the set
\[C_1=\{\xi<\kappa\st\xi\in\Pi^1_{\beta+2m}(\kappa)^+\}\]
is in $\R^1([\kappa]^{<\kappa})^*\subseteq\R^m(\Pi^1_\beta(\kappa))^*$. Therefore, $C_0\cap C_1\in\R^m(\Pi^1_\beta(\kappa))^*$.
\end{proof}

In fact, essentially the same proof shows that the additional statement in Corollary \ref{corollary_necessity} can be improved.

\begin{corollary}
Suppose $0<m<\omega$ and $\beta\in\{-1\}\cup\kappa$. If $\kappa\in\R^m(\Pi^1_\beta(\kappa))^+$ then the set
\[\{\xi<\kappa\st \xi\in\R_0(\R^{m-1}(\Pi^1_\beta(\kappa)))^+\cap\R^{m-1}(\Pi^1_{\beta+2}(\kappa))^+\}\]
is in $\R^m(\Pi^1_\beta(\kappa))^*$.
\end{corollary}

The following two corollaries of Theorem \ref{theorem_finite_ideal_diagram} show that the assumption ``$\exists\kappa$($\kappa\in\R^m(\Pi^1_{\beta+1}(\kappa))^+$)'' is strictly stronger than ``$\exists\kappa$($\kappa\in\R^m(\Pi^1_\beta(\kappa))^+$)''. In other words, each row of \Figure \ref{figure_finite_ideal_diagram} yields a strict hierarchy of large cardinals assuming the ideals are nontrivial.

\begin{corollary}
Suppose $0<m<\omega$, $\beta\in\{-1\}\cup\kappa$ and $\kappa\in\R^m(\Pi^1_\beta(\kappa))^+$. If $S\in\R^{\hat{m}}(\Pi^1_{\hat{\beta}}(\kappa))^+$ where $\hat{\beta}+2\hat{m}+1\leq \beta+2m$ then the set
\[T=\{\xi<\kappa\st S\cap\xi\in\R^{\hat{m}}(\Pi^1_{\hat{\beta}}(\xi))^+\}\]
is in $\R^m(\Pi^1_\beta(\kappa))^*$.
\end{corollary}

\begin{proof}
By Lemma \ref{lemma_complexity}, the fact that $S\in\R^{\hat{m}}(\Pi^1_{\hat{\beta}}(\kappa))^+$ is expressible by a $\Pi^1_{\hat{\beta}+2\hat{m}+1}$ sentence $\varphi$ over $(V_\kappa,\in,S)$. Since $\hat{\beta}+2\hat{m}+1\leq\beta+2m$ we have $\Pi^1_{\hat{\beta}+2\hat{m}+1}(\kappa)\subseteq\Pi^1_{\beta+2m}(\kappa)$, and by Corollary \ref{corollary_indescribability_in_finite_ramseyness}, since $\kappa\in\R^m(\Pi^1_\beta(\kappa))^+$ we have $\Pi^1_{\beta+2m}(\kappa)\subseteq\R^m(\Pi^1_\beta(\kappa))$, and thus the set
\[T=\{\xi<\kappa\st (V_\xi,\in,S\cap\xi)\models\varphi\}=\{\xi<\kappa\st S\cap\xi\in\R^{\hat{m}}(\Pi^1_{\hat{\beta}}(\xi))^+\}\]
is in $\Pi^1_{\hat{\beta}+2\hat{m}+1}(\kappa)^*\subseteq\R^m(\Pi^1_\beta(\kappa))^*$.
\end{proof}

\begin{corollary}\label{corollary_beta_to_beta_plus_one_hierarchy}
For all $0< m <\omega$ and all $\beta\in\{-1\}\cup\kappa$, if $\kappa\in\R^m(\Pi^1_{\beta+1}(\kappa))^+$ then the set
\[T=\{\xi<\kappa\st \xi\in\R^m(\Pi^1_\beta(\xi))^+\}\]
is in $\R^m(\Pi^1_{\beta+1}(\kappa))^*$.
\end{corollary}

Now let us show that the containments of the ideals from Theorem \ref{theorem_finite_ideal_diagram} as illustrated in \Figure \ref{figure_finite_ideal_diagram} are proper when the ideals involved are nontrivial.

\begin{theorem}\label{theorem_proper_containments_in_finite_diagram}
Suppose $0<m<\omega$ and $\beta<\kappa$.
\begin{enumerate}
\item If $\kappa\in\R^m(\Pi^1_{\beta+1}(\kappa))^+$ then $\R^m(\Pi^1_{\beta}(\kappa))\subsetneq\R^m(\Pi^1_{\beta+1}(\kappa))$.
\item If $\kappa\in\R^m(\Pi^1_{\beta}(\kappa))^+$ then $\R^{m-1}(\Pi^1_{\beta+2}(\kappa))\subsetneq\R^m(\Pi^1_\beta(\kappa))$.
\end{enumerate}
\end{theorem}

\begin{proof}
The containments follow from Theorem \ref{theorem_finite_ideal_diagram}, so we only need to show the properness of the containments.

For (1), let $S=\{\xi<\kappa\st\xi\in\R^m(\Pi^1_\beta(\xi))\}$. Then $S\in\R^m(\Pi^1_\beta(\kappa))^+$ by Lemma \ref{lemma_set_of_nons_is_positive}. Furthermore, by Corollary \ref{corollary_beta_to_beta_plus_one_hierarchy}, $S\in\R^m(\Pi^1_{\beta+1}(\kappa))$. Thus $\R^m(\Pi^1_{\beta}(\kappa))\subsetneq\R^m(\Pi^1_{\beta+1}(\kappa))$. 

For (2), let $S=\{\xi<\kappa\st\xi\in\R^{m-1}(\Pi^1_{\beta+2}(\xi))\}$. By Lemma \ref{lemma_set_of_nons_is_positive} we see that $S\in\R^{m-1}(\Pi^1_{\beta+2}(\kappa))^+$. From Corollary \ref{corollary_ramseyness_reflects_indescribability}, it follows that $\kappa\setminus S\in\R^m([\kappa]^{<\kappa})^*$ and since $\R^m([\kappa]^{<\kappa})\subseteq\R^m(\Pi^1_\beta(\kappa))$, this implies $S\in\R^m(\Pi^1_\beta(\kappa))$.
\end{proof}

The next corollary, which follows directly from Theorem \ref{theorem_finite_ideal_diagram}, shows that iterating the Ramsey operator on an indescribability ideal $\Pi^1_{\beta+n}(\kappa)$ infinitely many times leads to the same ideal, no matter what $n<\omega$ was initially chosen (see \Figure \ref{figure_first_collapse}).

\begin{corollary}\label{corollary_collapse}
The following hold.
\begin{enumerate}
\item If $\kappa\in\R^\omega([\kappa]^{<\kappa})^+$, then for all $n<\omega$ we have \[\R^\omega([\kappa]^{<\kappa})=\R^\omega(\Pi^1_n(\kappa)).\]
\item For all limit ordinals $\beta<\kappa$, if $\kappa\in\R^\omega(\Pi^1_\beta(\kappa))^+$, then for all $n<\omega$ we have \[\R^\omega(\Pi^1_\beta(\kappa))=\R^\omega(\Pi^1_{\beta+n}(\kappa)).\]
\end{enumerate}
\end{corollary}

\begin{figure}
\centering
\begin{tikzpicture}[x=2.5cm,y=0.7cm]
\tiny
\node (pi1-1) at (0,0) {$\Pi^1_\beta(\kappa)$};
\node (pi10) at (1,0) {$\Pi^1_{\beta+1}(\kappa)$};
\node (pi11) at (2,0) {$\Pi^1_{\beta+2}(\kappa)$};
\node (pi13) at (2.7,0) {$\cdots$};

\node (rpi1-1) at (0,1) {$\R(\Pi^1_\beta(\kappa))$};
\node (rpi10) at (1,1) {$\R(\Pi^1_{\beta+1}(\kappa))$};
\node (rpi11) at (2,1) {$\R(\Pi^1_{\beta+2}(\kappa))$};
\node (rpi13) at (2.7,1) {$\cdots$};

\node (r2pi1-1) at (0,2) {$\R^2(\Pi^1_\beta(\kappa))$};
\node (r2pi10) at (1,2) {$\R^2(\Pi^1_{\beta+1}(\kappa))$};
\node (r2pi11) at (2,2) {$\R^2(\Pi^1_{\beta+2}(\kappa))$};
\node (r2pi13) at (2.7,2) {$\cdots$};

\node (r3pi1-1) at (0,3) {$\R^3(\Pi^1_\beta(\kappa))$};
\node (r3pi10) at (1,3) {$\R^3(\Pi^1_{\beta+1}(\kappa))$};
\node (r3pi11) at (2,3) {$\R^3(\Pi^1_{\beta+2}(\kappa))$};
\node (r3pi13) at (2.7,3) {$\cdots$};

\node (r4pi1-1) at (0,4) {$\vdots$};
\node (r4pi10) at (1,4) {$\vdots$};
\node (r4pi11) at (2,4) {$\vdots$};

\draw[rounded corners=0.1cm] (-0.5,-0.7) rectangle (2.9,4.5);

\node (rwpi1-1) at (1,6) {$\R^\omega(\Pi^1_\beta(\kappa))=\R^\omega(\Pi^1_{\beta+1}(\kappa))=\R^\omega(\Pi^1_{\beta+2}(\kappa))=\cdots$};

\draw [->] (r4pi1-1.north) -- (rwpi1-1);
\draw [->] (r4pi10.north) -- (rwpi1-1);
\draw [->] (r4pi11.north) -- (rwpi1-1);

\node (rw+1pi1-1) at (1,7.5) {$\R^{\omega+1}(\Pi^1_\beta(\kappa))$};
\node (rw+2pi1-1) at (1,9) {$\R^{\omega+2}(\Pi^1_\beta(\kappa))$};
\node (rw+3pi1-1) at (1,10.5) {};
\node (rw+4pi1-1) at (1,11) {$\vdots$};

\draw [->] (rwpi1-1) -- (rw+1pi1-1);
\draw [->] (rw+1pi1-1) -- (rw+2pi1-1);
\draw [->] (rw+2pi1-1) -- (rw+3pi1-1);



\end{tikzpicture}

\caption{\tiny Indescribability becomes redundant as one moves up the Ramsey hierarchy.}\label{figure_first_collapse}
\end{figure}
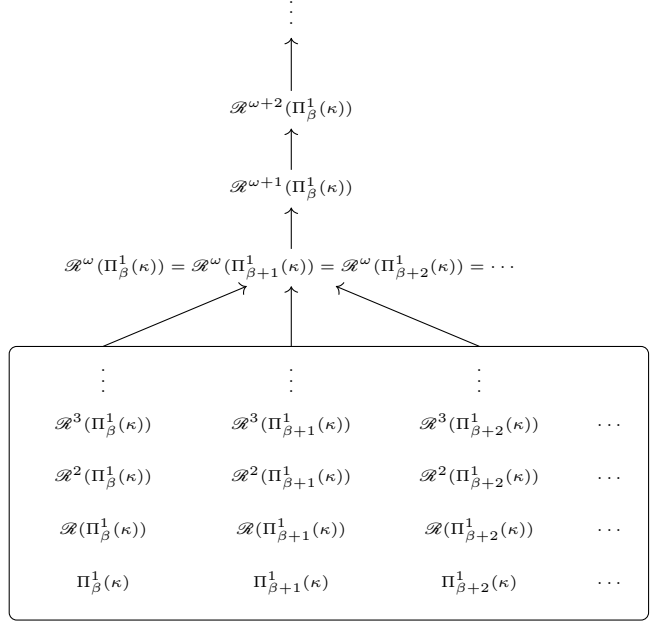

Note that, although Corollary \ref{corollary_collapse} is an easy consequence of Theorem \ref{theorem_finite_ideal_diagram}, its proof is substantially different from that of the observation $\R^\omega([\kappa]^{<\kappa})=\R^\omega(\NS_\kappa)$ made above in Remark \ref{remark_fengs_defintion}, because the ideals involved \emph{do not} fit into a chain.

\begin{remark}\label{remark_first_collapse}
Let us point out an easy consequence of Corollary \ref{corollary_collapse} which will serve as motivation for some of the results in Section \ref{section_indescribability_in_infinite_ramseyness} (see Remark \ref{remark_hierarchy}). The previous corollary easily implies that when $\omega\leq\alpha<\kappa$ and $\beta<\kappa$, the assertion $\kappa\in\R^\alpha(\Pi^1_{\beta}(\kappa))^+$ is equivalent to $\kappa\in\R^\alpha(\Pi^1_{\beta+n}(\kappa))^+$ for all $n<\omega$. In other words, for $\omega\leq\alpha<\kappa$, $\kappa$ being $\alpha$-$\Pi^1_\beta$-Ramsey is equivalent to $\kappa$ being $\alpha$-$\Pi^1_{\beta+n}$-Ramsey for all $n<\omega$.
\end{remark}

\section{Indescribability in infinite degrees of Ramseyness}\label{section_indescribability_in_infinite_ramseyness}


We now proceed to extend some of the results of Section \ref{section_indescribability_in_finite_ramseyness} to the ideals $\R^\alpha(\Pi^1_\beta(\kappa))$ for $\alpha>\omega$.

\begin{lemma}\label{lemma_indescribability_from_ramseyness}
For all ordinals $\alpha,\beta<\kappa$ the following hold.
\begin{enumerate}
\item If $\alpha$ is a successor ordinal then $\Pi^1_{\beta+\alpha}(\kappa)\subseteq\R^\alpha(\Pi^1_\beta(\kappa))$.
\item If $\alpha$ is a limit ordinal or if $\alpha=0$ then $\bigcup_{\xi<\beta+\alpha}\Pi^1_\xi(\kappa)\subseteq\R^\alpha(\Pi^1_\beta(\kappa))$.
\end{enumerate}
\end{lemma}

\begin{proof}
Fix an ordinal $\beta<\kappa$. We proceed by induction on $\alpha$. Clearly the result holds for $\alpha=0$, since $\bigcup_{\xi<\beta}\Pi^1_\xi(\kappa)\subseteq\Pi^1_\beta(\kappa)$. The case in which $\alpha$ is a limit is trivial.

For the successor step of the induction, let us argue that $\R^{\alpha+1}(\Pi^1_\beta(\kappa))^+\subseteq\Pi^1_{\beta+\alpha+1}(\kappa)^+$, assuming the result holds for $\alpha$. Suppose $X\in\R^{\alpha+1}(\Pi^1_\beta(\kappa))^+$. Then every regressive function $f:[X]^{<\omega}\to \kappa$ has a homogeneous set $H\in\R^\alpha(\Pi^1_\beta(\kappa))^+$. By our inductive hypothesis, $\R^\alpha(\Pi^1_\beta(\kappa))^+\subseteq\bigcap_{\xi<\beta+\alpha}\Pi^1_\xi(\kappa)^+$. Thus, by Lemma \ref{lemma_baumgartner_bagaria}, it follows that $X\in\Pi^1_{\beta+\alpha+1}(\kappa)^+$.
\end{proof}

Among other things, the next theorem shows that Lemma \ref{lemma_indescribability_from_ramseyness} (1) can be improved when $\alpha$ is a successor ordinal which is not an immediate successor of a limit ordinal.

\begin{theorem}\label{theorem_indescribability_in_infinite_ramseyness} 
Suppose $\kappa$ is a cardinal, $\alpha<\kappa$ is a limit ordinal and $\beta\in\{-1\}\cup\kappa$. For all $m<\omega$, if $\kappa\in\R^{\alpha+m+1}(\Pi^1_\beta(\kappa))^+$ then
\[\R^{\alpha+m+1}(\Pi^1_\beta(\kappa))=\overline{\R_0(\R^{\alpha+m}(\Pi^1_\beta(\kappa)))\cup\Pi^1_{\beta+\alpha+2m+1}(\kappa)}.\]
\end{theorem}

\begin{proof}
For the base case, let $m=0$. Let $I=\overline{\R_0(\R^\alpha(\Pi^1_\beta(\kappa)))\cup\Pi^1_{\beta+\alpha+1}(\kappa)}$. We will show that $X\in\R^{\alpha+1}(\Pi^1_\beta(\kappa))^+$ if and only if $X\in I^+$. 

Suppose $X\in\R^{\alpha+1}(\Pi^1_\beta(\kappa))^+$. By Remark \ref{remark_ideal_generated}, it suffices to show that $X\in\Pi^1_{\beta+\alpha+1}(\kappa)^+$ and $X\in\R_0(\R^\alpha(\Pi^1_\beta(\kappa)))^+$. By Lemma \ref{lemma_indescribability_from_ramseyness}, we have $X\in\Pi^1_{\beta+\alpha+1}(\kappa)^+$. Let us show that $X\in\R_0(\R^\alpha(\Pi^1_\beta(\kappa)))^+$. Fix a regressive function $f:[X]^{<\omega}\to \kappa$ and a club $C\subseteq\kappa$. By assumption there is a set $H\in\R^\alpha(\Pi^1_\beta(\kappa))^+$ homogeneous for $f$. By Lemma \ref{lemma_complexity}, the fact that $H\in\R^\alpha(\Pi^1_\beta(\kappa))^+$ can be expressed by a $\Pi^1_{\beta+\alpha}$ sentence $\varphi$ over $(V_\kappa,\in,H)$, and since $X\cap C\in\Pi^1_{\beta+\alpha+1}(\kappa)^+$, there is a $\xi\in X\cap C$ with $\xi>\alpha,\beta$ such that $(V_\xi,\in,H\cap\xi)\models\varphi$, and hence  $H\cap\xi\in\R^\alpha(\Pi^1_\beta(\xi))^+$. Thus, $X\in\R_0(\R^\alpha(\Pi^1_\beta(\kappa)))^+$.

Now suppose $X\in I^+$. We argue that $X\in\R^{\alpha+1}(\Pi^1_\beta(\kappa))^+$. Let $f:[X]^{<\omega}\to \kappa$ be a regressive function. Suppose that every homogeneous set $H$ for $f$ is in $\R^\alpha(\Pi^1_\beta(\kappa))$. By Lemma \ref{lemma_complexity}, this can be expressed by a $\Pi^1_{\beta+\alpha+1}$ sentence $\varphi$ over $(V_\kappa,\in,f)$. This implies that the set $C=\{\xi<\kappa\st(V_\xi,\in,f\cap V_\xi)\models\varphi\}$ is in $\Pi^1_{\beta+\alpha+1}(\kappa)^*$. Since $X\in I^+$, it follows that $X$ is not the union of a set in $\R_0(\R^\alpha(\Pi^1_\beta(\kappa)))$ and a set in $\Pi^1_{\beta+\alpha+1}$. Since $X=(X\cap C)\cup(X\setminus C)$ and $X\setminus C\in\Pi^1_{\beta+\alpha+1}(\kappa)$, we see that $X\cap C\in\R_0(\R^\alpha(\Pi^1_\beta(\kappa)))^+$. Hence there is a $\xi\in X\cap C$ with $\xi>\alpha,\beta$ for which there is a set $H\subseteq X\cap C\cap\xi$ in $\R^\alpha(\Pi^1_\beta(\xi))^+$ homogeneous for $f$. This contradicts $\xi\in C$.

For the inductive step, we suppose 
\[\R^{\alpha+m}(\Pi^1_\beta(\kappa))=\overline{\R_0(\R^{\alpha+m-1}(\Pi^1_\beta(\kappa)))\cup\Pi^1_{\beta+\alpha+2m-1}(\kappa)}\]
and show
\[\R^{\alpha+m+1}(\Pi^1_\beta(\kappa))=\overline{\R_0(\R^{\alpha+m}(\Pi^1_\beta(\kappa)))\cup\Pi^1_{\beta+\alpha+2m+1}(\kappa)}.\]
Let $I=\overline{\R_0(\R^{\alpha+m}(\Pi^1_\beta(\kappa)))\cup\Pi^1_{\beta+\alpha+2m+1}(\kappa)}$. 

Suppose $X\in\R^{\alpha+m+1}(\Pi^1_\beta(\kappa))^+$. By Remark \ref{remark_ideal_generated}, it suffices to show that $X\in \Pi^1_{\beta+\alpha+2m+1}(\kappa)^+$ and $X\in \R_0(\R^{\alpha+m}(\Pi^1_\beta(\kappa)))^+$. By our inductive hypothesis $Q:=\R^{\alpha+m}(\Pi^1_\beta(\kappa))^+\subseteq\Pi^1_{\beta+\alpha+2m-1}(\kappa)^+$, and thus by Lemma \ref{lemma_baumgartner_bagaria}  we have $X\in\Pi^1_{\beta+\alpha+2m+1}(\kappa)^+$. Let us show that $X\in\R_0(\R^{\alpha+m}(\Pi^1_\beta(\kappa)))^+$. Fix a regressive function $f:[X]^{<\omega}\to \kappa$ and a club $C\subseteq\kappa$. Since $X\in\R^{\alpha+m+1}(\Pi^1_\beta(\kappa))^+$ there is a set $H\in\R^{\alpha+m}(\Pi^1_\beta(\kappa))^+$ homogeneous for $f$. By Lemma \ref{lemma_complexity}, the fact that $H\in\R^{\alpha+m}(\Pi^1_\beta(\kappa))^+$ can be expressed by a $\Pi^1_{\beta+\alpha+2m}$ sentence $\varphi$ over $(V_\kappa,\in,H)$. Since $X\cap C\in\Pi^1_{\beta+\alpha+2m+1}(\kappa)^+$, there is a $\xi\in X\cap C$ with $\xi>\alpha+m,\beta$, for which $H\cap\xi\in\R^{\alpha+m}(\Pi^1_\beta(\xi))^+$. Thus $X\in\R_0(\R^{\alpha+m}(\Pi^1_\beta(\kappa)))^+$.

Conversely, suppose $X\in I^+$. Let $f:[X]^{<\omega}\to \kappa$ be a regressive function. Suppose that every set which is homogeneous for $f$ is in $\R^{\alpha+m}(\Pi^1_\beta(\kappa))$. By Lemma \ref{lemma_complexity}, this can be expressed by a $\Pi^1_{\beta+\alpha+2m+1}$ sentence $\varphi$ over $(V_\kappa,\in,f)$. Thus the set $C=\{\xi<\kappa\st (V_\xi,\in,f\cap V_\xi)\models\varphi\}$ is in $\Pi^1_{\beta+\alpha+2m+1}(\kappa)^*$. Since $X\in I^+$, it follows that $X$ is not the union of a set in $\R_0(\R^{\alpha+m}(\Pi^1_\beta(\kappa)))$ and a set in $\Pi^1_{\beta+\alpha+2m+1}(\kappa)$, and since $X\setminus C\in\Pi^1_{\beta+\alpha+2m+1}(\kappa)$, we see that $X\cap C\in\R_0(\R^{\alpha+m}(\Pi^1_\beta(\kappa)))^+$. Hence there is a $\xi\in X\cap C$ with $\xi>\alpha+m,\beta$ such that there is a set $H\subseteq X\cap C\cap\xi$ in $\R^{\alpha+m}(\Pi^1_\beta(\xi))^+$ homogeneous for $f$. This contradicts $\xi\in C$.
\end{proof}

\begin{remark}\label{remark_hierarchy}
We would like to use Theorem \ref{theorem_indescribability_in_infinite_ramseyness} to prove an analogue of Corollary \ref{corollary_beta_to_beta_plus_one_hierarchy}, which would say that the strength of the hypothesis ``$\exists\kappa$ $\kappa\in\R^\alpha(\Pi^1_\beta(\kappa))$'' increases as $\beta$ increases. However, there is an added complication, as illustrated in Corollary \ref{corollary_collapse} and Remark \ref{remark_first_collapse}, which is that even if $\beta_0<\beta_1<\kappa$, it may be that $\kappa\in\R^\alpha(\Pi^1_{\beta_0}(\kappa))$ is equivalent to $\kappa\in\R^\alpha(\Pi^1_{\beta_1}(\kappa))$, if $\alpha$ is large enough. Thus, in order to show that the hypotheses $\kappa\in\R^\alpha(\Pi^1_\beta(\kappa))$ form a hierarchy as $\beta$ increases, we will need to determine at what $\alpha$ do the hypotheses $\kappa\in\R^\alpha(\Pi^1_{\beta_0}(\kappa))$ and $\kappa\in\R^\alpha(\Pi^1_{\beta_1}(\kappa))$ become equivalent.
\end{remark}

\begin{remark}
Using Theorem \ref{theorem_indescribability_in_infinite_ramseyness}, it is possible to formulate a characterization of $\kappa\in\R^{\alpha+m+1}(\Pi^1_\beta(\kappa))^+$ in terms of the relevant ideals along the lines of Corollary \ref{corollary_necessity} above. Moreover, one can show that reference to the ideals in such a characterization is, in fact, necessary. We leave the details to the reader.
\end{remark}

Let us prove Theorem 1.2 mentioned in Section \ref{section_introduction}. That is, we will show that for any two ordinals $\beta_0<\beta_1<\kappa$, the two increasing chains of ideals $\<\R^\alpha(\Pi^1_{\beta_0}(\kappa))\st\alpha<\kappa\>$ and $\<\R^\alpha(\Pi^1_{\beta_1}(\kappa))\st\alpha<\kappa\>$ are eventually equal, and we determine the precise index at which the equality begins (see \Figure \ref{figure_culmination} for an illustration of this result).

\begin{theorem_intro}
Suppose $\beta_0<\beta_1$ are in $\{-1\}\cup\kappa$ and let $\sigma=\ot(\beta_1\setminus\beta_0)$. Define $\alpha=\sigma\cdot\omega$. Suppose $\kappa\in\R^\alpha(\Pi^1_{\beta_1}(\kappa))^+$ so that the ideals under consideration are nontrivial. Then $\alpha$ is the least ordinal such that $\R^\alpha(\Pi^1_{\beta_0}(\kappa))=\R^\alpha(\Pi^1_{\beta_1}(\kappa))$.
\end{theorem_intro}

\begin{proof}
First, let us show $\R^\alpha(\Pi^1_{\beta_0}(\kappa))=\R^\alpha(\Pi^1_{\beta_1}(\kappa))$. Since $\beta_0<\beta_1$, it is clear that $\R^\alpha(\Pi^1_{\beta_0}(\kappa))\subseteq\R^\alpha(\Pi^1_{\beta_1}(\kappa))$. Let us show that $\R^\alpha(\Pi^1_{\beta_0}(\kappa))\supseteq\R^\alpha(\Pi^1_{\beta_1}(\kappa))$. If $\sigma=\ot(\beta_1\setminus\beta_0)=n$ is finite then $\alpha=\omega$ and the result follows from Corollary \ref{corollary_collapse} since $ \R^\omega(\Pi^1_{\beta_0}(\kappa))=\R^\omega(\Pi^1_{\beta_0+n}(\kappa))$. Suppose $\sigma=\ot(\beta_1\setminus\beta_0)\geq\omega$. Then $\alpha=\sigma\cdot\omega$ is a limit of limits. Let us show that $\R^\xi(\Pi^1_{\beta_1}(\kappa))\subseteq\R^\alpha(\Pi^1_{\beta_0}(\kappa))$ for each limit $\xi<\alpha$. Fix a limit ordinal $\xi<\alpha$. For ordinals $\zeta<\alpha$ we define $i(\zeta)$ to be the least $i<\omega$ such that $\zeta<\sigma\cdot i$. Notice that if we let $\gamma$ be the greatest limit ordinal which is less than or equal to $\sigma$ then $\beta_1=\beta_0+\sigma\leq\beta_0+\gamma+2m+1<\beta_0+\sigma\cdot 2$ for some odd natural number $2m+1<\omega$. Now, by Theorem \ref{theorem_indescribability_in_infinite_ramseyness}, we have 
\begin{align}
\Pi^1_{\beta_1}(\kappa)\subseteq\Pi^1_{\beta_0+\gamma+2m+1}(\kappa)\subseteq\R^{\gamma+m+1}(\Pi^1_{\beta_0}(\kappa)).\label{equation_ind}
\end{align} 
Applying the Ramsey operator $\xi$ times to (\ref{equation_ind}) yields
\[\R^\xi(\Pi^1_{\beta_1}(\kappa))\subseteq\R^{\gamma+m+1+\xi}(\Pi^1_{\beta_0}(\kappa))\]
Since $i(\gamma)+i(\xi)<\omega$ it follows that $\gamma+m+1+\xi=\gamma+\xi<\sigma\cdot i(\gamma)+\sigma\cdot i(\xi)$ must be less than $\alpha$. Thus $\R^\xi(\Pi^1_{\beta_1}(\kappa))\subseteq\R^{\alpha}(\Pi^1_{\beta_0}(\kappa))$.

Next, let us show that if $\hat{\alpha}<\alpha$ then $\R^{\hat{\alpha}}(\Pi^1_{\beta_0}(\kappa))\subsetneq\R^{\hat{\alpha}}(\Pi^1_{\beta_1}(\kappa))$. If $\sigma=\ot(\beta_1\setminus\beta_0)$ is finite, in which case $\alpha=\omega$, then the result follows from Theorem \ref{theorem_proper_containments_in_finite_diagram}. On the other hand, if $\sigma$ is infinite, then $\alpha=\sigma\cdot\omega>\omega$ and $\alpha$ is a limit of limits. Let $\bar{\alpha}$ be a limit ordinal with $\hat{\alpha}<\bar{\alpha}+1<\alpha$. It suffices to show that $\R^{\bar{\alpha}+1}(\Pi^1_{\beta_0}(\kappa))\subsetneq\R^{\bar{\alpha}+1}(\Pi^1_{\beta_1}(\kappa))$. Let 
\[S=\{\xi<\kappa\st\xi\in\R^{\bar{\alpha}+1}(\Pi^1_{\beta_0}(\xi))\}.\]
Since $\kappa\in\R^\alpha(\Pi^1_{\beta_1}(\kappa))^+$, it follows from Lemma \ref{lemma_set_of_nons_is_positive} that $S\notin\R^{\bar{\alpha}+1}(\Pi^1_{\beta_0}(\kappa))$. Furthermore, by Lemma \ref{lemma_complexity}, the fact that $S\notin\R^{\bar{\alpha}+1}(\Pi^1_{\beta_0}(\kappa))$ is $\Pi^1_{\beta_0+\bar{\alpha}+2}$-expressible over $V_\kappa$ and so the set $C=\kappa\setminus S$ is in $\Pi^1_{\beta_0+\bar{\alpha}+2}(\kappa)^*$.
By Theorem \ref{theorem_indescribability_in_infinite_ramseyness}, $\Pi^1_{\beta_0+\sigma+\bar{\alpha}+1}(\kappa)\subseteq\R^{\bar{\alpha}+1}(\Pi^1_{\beta_1}(\kappa))$. Since $\bar{\alpha}<\alpha=\sigma\cdot\omega$, it follows that $\beta_0+\bar{\alpha}+2<\beta_0+\sigma+\bar{\alpha}+1$ and thus $\Pi^1_{\beta_0+\bar{\alpha}+2}(\kappa)\subseteq\R^{\bar{\alpha}+1}(\Pi^1_{\beta_1}(\kappa))$. This implies that $C\in\R^{\bar{\alpha}+1}(\Pi^1_{\beta_1}(\kappa))^*$ and thus $S\in \R^{\bar{\alpha}+1}(\Pi^1_{\beta_1}(\kappa))$.
\end{proof}

\begin{figure}
\centering
\begin{tikzpicture}[x=0.18cm,y=0.18cm]

\tiny

\foreach \i in {0,10,20,40} {

\foreach \x in {1,...,6} {
	\foreach \y in {1,...,6} {
		\node[circle,draw=black, fill=black, inner sep=0pt,minimum size=1pt] (a\i\x\y) at ({6*(6-6*pow(\x+1,-0.1))-1.5+\i},{6*(6-6*pow(\y+1,-0.1))-1.5}) {};
	}
}
\draw[rounded corners=0.1cm] (\i,0) rectangle (5.7+\i,5.7);


\node[circle,draw=black, fill=white, inner sep=0pt,minimum size=3pt] (w\i) at ({6*(6-6*pow(3+1,-0.1))-1.5+\i-0.3},7.5) {};

\foreach \j in {1,2.65,6} {
	\draw ({6*(6-6*pow(\j+1,-0.1))-1.5+\i},6) -- (w\i);
}

\foreach \y in {2,...,6} {
	\node[circle,draw=black, fill=black, inner sep=0pt,minimum size=1pt] (w\i\y) at ({6*(6-6*pow(3+1,-0.1))-1.5+\i-0.3},{7.5+6*(6-6*pow(\y+1,-0.1))-2.5}) {};
}

\node[circle,draw=black, fill=white, inner sep=0pt,minimum size=3pt] (w\i7) at ({6*(6-6*pow(3+1,-0.1))-1.5+\i-0.3},{7.5+6*(6-6*pow(7+1,-0.1))-2.5}) {};

\draw (w\i)--(w\i7);

\node (ell\i) at ({6*(6-6*pow(3+1,-0.1))-1.5+\i-0.3},15) {$\vdots$};

\draw[rounded corners=0.1cm] (-1,-1) rectangle (26.7,17);

\node[circle,draw=black, fill=white, inner sep=0pt,minimum size=3pt] (lima) at ({6*(6-6*pow(3+1,-0.1))-1.5+10-0.3},22) {};

\foreach \k in {0,10,20} {
	\draw ({6*(6-6*pow(3+1,-0.1))-1.5+\k-0.3},18) -- (lima);
}

\node[circle,draw=black, fill=white, inner sep=0pt,minimum size=3pt] (limb) at ({6*(6-6*pow(3+1,-0.1))-1.5+40-0.3},17) {};

\foreach \y in {2,...,6} {
	\node[circle,draw=black, fill=black, inner sep=0pt,minimum size=1pt] (lim\y) at ({6*(6-6*pow(3+1,-0.1))-1.5+40-0.3},{7.5+6*(6-6*pow(\y+1,-0.1))-2.5+10}) {};
}

\node[circle,draw=black, fill=white, inner sep=0pt,minimum size=3pt] (lim7) at ({6*(6-6*pow(3+1,-0.1))-1.5+40-0.3},{7.5+6*(6-6*pow(7+1,-0.1))-2.5+10}) {};

\draw (limb) -- (lim7);


\foreach \y in {2,...,6} {
	\node[circle,draw=black, fill=black, inner sep=0pt,minimum size=1pt] (up10\y) at ({6*(6-6*pow(3+1,-0.1))-1.5+10-0.3},{7.5+6*(6-6*pow(\y+1,-0.1))-2.5+10+5}) {};
}

\node[circle,draw=black, fill=white, inner sep=0pt,minimum size=3pt] (up107) at ({6*(6-6*pow(3+1,-0.1))-1.5+10-0.3},{7.5+6*(6-6*pow(7+1,-0.1))-2.5+10+5}) {};

\foreach \y in {2,...,6} {
	\node[circle,draw=black, fill=black, inner sep=0pt,minimum size=1pt] (up40\y) at ({6*(6-6*pow(3+1,-0.1))-1.5+40-0.3},{7.5+6*(6-6*pow(\y+1,-0.1))-2.5+10+5}) {};
}

\node[circle,draw=black, fill=white, inner sep=0pt,minimum size=3pt] (up407) at ({6*(6-6*pow(3+1,-0.1))-1.5+40-0.3},{7.5+6*(6-6*pow(7+1,-0.1))-2.5+10+5}) {};

\draw (lim7) -- (up407);
\draw (lima) -- (up107);

\node (fdots10) at ({6*(6-6*pow(3+1,-0.1))-1.5+10-0.3},{7.5+6*(6-6*pow(7+1,-0.1))-2.5+10+5+3}) {$\vdots$};

\node (fdots40) at ({6*(6-6*pow(3+1,-0.1))-1.5+40-0.3},{7.5+6*(6-6*pow(7+1,-0.1))-2.5+10+5+3}) {$\vdots$};

\node[circle,draw=black, fill=white, inner sep=0pt,minimum size=3pt] (fnode) at (28,38) {};

\draw (fdots10.north) -- (fnode);
\draw (fdots40.north) -- (fnode);

\node[minimum size=0pt] (upperbound) at (58,38) {};
\node (middle) at (55,19) {$\alpha=\ot(\beta_1\setminus\beta_0)\cdot\omega$};
\node (lowerbound) at (58,0) {};

\draw (53,38) -- (upperbound);
\draw (53,0) -- (lowerbound);
\draw[->] (middle) -- (55,0);
\draw[->] (middle) -- (55,38);

}

\node (beta0) at ({6*(6-6*pow(2+1,-0.1))-1.5+10},{6*(6-6*pow(1+1,-0.1))-1.5-4}) {$\beta_0$};

\draw (a1021) -- (a1026);
\draw (a4031) -- (a4036);

\node (beta1) at ({6*(6-6*pow(3+1,-0.1))-1.5+40},{6*(6-6*pow(1+1,-0.1))-1.5-4}) {$\beta_1$};

\node (poop) at (28,40) {$\R^\alpha(\Pi^1_{\beta_0}(\kappa))=\R^\alpha(\Pi^1_{\beta_1}(\kappa))$};

\end{tikzpicture}

\caption{\tiny For $\beta_0,\beta_1<\kappa$ the ideal chains $\<\R^\alpha(\Pi^1_{\beta_0}(\kappa))\st\alpha<\kappa\>$ and $\<\R^\alpha(\Pi^1_{\beta_1}(\kappa))\st\alpha<\kappa\>$ become equal at $\alpha=\ot(\beta_1\setminus\beta_0)\cdot\omega$.}\label{figure_culmination}
\end{figure}
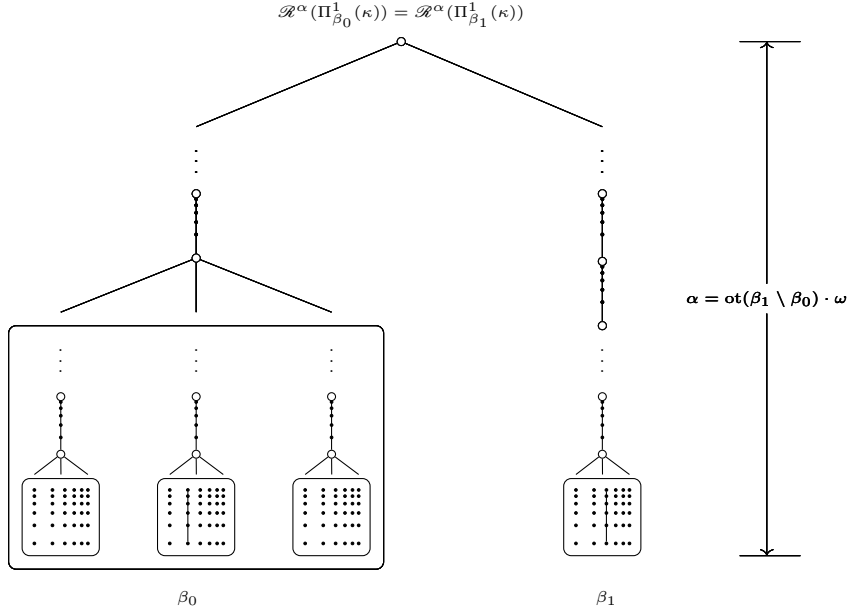

\begin{corollary}\label{corollary_main_redundnacy}
If $\kappa\in \R^\kappa([\kappa]^{<\kappa})^+$ then for all $\beta_0,\beta_1<\kappa$, assuming the ideals involved are nontrivial, we have
\[\R^\kappa(\Pi^1_{\beta_0}(\kappa))=\R^\kappa(\Pi^1_{\beta_1}(\kappa)).\]
\end{corollary}

As a direct corollary of Theorem \ref{theorem_culmination} we derive the following, which is the analogue of Theorem \ref{theorem_proper_containments_in_finite_diagram} (1) for the ideals $\R^\alpha(\Pi^1_\beta(\kappa))$ when $\alpha>\omega$.\footnote{Below we will derive the analogue of Theorem \ref{theorem_proper_containments_in_finite_diagram} (2) for $\alpha>\omega$ as a consequence of Theorem \ref{theorem_infinite_ideal_diagram}.}

\begin{corollary}
Suppose $\beta_0<\beta_1$ are in $\{-1\}\cup\kappa$. If $\alpha<\ot(\beta_1\setminus\beta_0)\cdot\omega$ and $\kappa\in\R^\alpha(\Pi^1_{\beta_1}(\kappa))^+$, then $\R^\alpha(\Pi^1_{\beta_0}(\kappa))\subsetneq\R^\alpha(\Pi^1_{\beta_1}(\kappa))$.
\end{corollary}

Next, we show that for $\omega\leq\alpha<\kappa$ and $\beta_0<\beta_1<\kappa$, the hypothesis $\kappa\in\R^\alpha(\Pi^1_{\beta_1}(\kappa))^+$ implies that there are many $\xi<\kappa$ which satisfy $\xi\in\R^\alpha(\Pi^1_{\beta_0}(\xi))^+$, \emph{assuming $\beta_0$ and $\beta_1$ are far enough apart}. Thus, the hypotheses of the form $\kappa\notin\R^\alpha(\Pi^1_\beta(\kappa))$ provide a strictly increasing refinement of Feng's original hierarchy (see \Figure \ref{figure_a_refinement_of_the_ramsey_hierarchy}). 

\begin{theorem}\label{theorem_hierarchy_result_for_infinite_alpha}
Suppose $\beta_0<\beta_1$ are in $\{-1\}\cup\kappa$ and $\alpha<\ot(\beta_1\setminus\beta_0)\cdot\omega$. If $\kappa\in\R^\alpha(\Pi^1_{\beta_1}(\kappa))^+$ then the set
\[\{\xi<\kappa\st\xi\in\R^\alpha(\Pi^1_{\beta_0}(\xi))^+\}\]
is in $\R^\alpha(\Pi^1_{\beta_1}(\kappa))^*$.
\end{theorem}

\begin{proof}
Suppose $\alpha$ is a successor. That is, $\alpha=\bar{\alpha}+m+1$ where $\bar{\alpha}$ is a limit ordinal and $m<\omega$. Since $\kappa\in\R^\alpha(\Pi^1_{\beta_1}(\kappa))^+$ and $\beta_0<\beta_1$ we have $\kappa\in\R^\alpha(\Pi^1_{\beta_0}(\kappa))^+$, which is expressible by a $\Pi^1_{\beta_0+\bar{\alpha}+2(m+1)}$ sentence $\varphi$ by Lemma \ref{lemma_complexity}. Since $\alpha=\bar{\alpha}+m+1<\ot(\beta_1\setminus\beta_0)\cdot\omega$ it follows that $\beta_0+\bar{\alpha}+2(m+1)<\beta_1+\bar{\alpha}+2m+1$.\footnote{This is because $\ot(\beta_1\setminus\beta_0)\cdot\omega$ is a limit ordinal, and thus adding any finite number of copies of $\ot(\beta_1\setminus\beta_0)$ to $\beta_0+\bar{\alpha}+2(m+1)$ will produce an ordinal which is less than $\beta_1$.} Now by Theorem \ref{theorem_indescribability_in_infinite_ramseyness}, we see that $\Pi^1_{\beta_1+\bar{\alpha}+2m+1}(\kappa)\subseteq\R^\alpha(\Pi^1_{\beta_1}(\kappa))$, and thus, the set
\[C=\{\xi<\kappa\st (V_\xi,\in)\models\varphi\}=\{\xi<\kappa\st\xi\in\R^\alpha(\Pi^1_{\beta_0}(\kappa))\}\]
is in $\R^\alpha(\Pi^1_{\beta_1}(\kappa))^*$. 

The fact that the result holds for successors easily implies that it holds for limits.
\end{proof}


\begin{figure}
\centering
\begin{tikzpicture}[x=0.75cm,y=0.75cm]
\tiny

\node (elli) at (0,-0.2) {};

\node[draw, rounded corners=0.1cm] (rwpi1-1) at (0,1) 
{$\Pi_\alpha$-Ramsey};

\node (rwpi1w) at (0,2) 
{$\alpha$-$\Pi^1_{\beta_0}$-Ramsey};

\node[right] (rwpi1wextra) at (rwpi1w.east) {$\longleftarrow$ choose $\beta_0<\kappa$ such that $\alpha<\ot(\beta_0)\cdot\omega$
};

\node (rwpi1w*2) at (0,3) 
{$\alpha$-$\Pi^1_{\beta_1}$-Ramsey};

\node[right] (rwpi1w*2extra) at (rwpi1w*2.east) {$\longleftarrow$ choose $\beta_1<\kappa$ such that $\alpha<\ot(\beta_1\setminus\beta_0)\cdot\omega$};

\node[draw, rounded corners=0.1cm] (rw+1pi1-1) at (0,4.5) {$\Pi_{\alpha+1}$-Ramsey};

\node (rw+1pi1w) at (0,6) {};

\draw[-...] (elli)--(rwpi1-1);
\draw (rwpi1-1)--(rwpi1w);
\draw (rwpi1w)--(rwpi1w*2);
\draw[-...] (rwpi1w*2)--(rw+1pi1-1);
\draw[-...] (rw+1pi1-1)--(rw+1pi1w);

\end{tikzpicture}
\caption{\tiny If $\omega\leq\alpha<\kappa$ or if $\alpha<\omega$ is even, by choosing $\beta$'s appropriately, the $\alpha$-$\Pi^1_\beta$-cardinals yield a hierarchy of hypotheses strictly between Feng's \cite{MR1077260} $\Pi_\alpha$-Ramsey and $\Pi_{\alpha+1}$-Ramsey cardinals (see Theorem \ref{theorem_hierarchy_result_for_infinite_alpha}).}\label{figure_a_refinement_of_the_ramsey_hierarchy}
\end{figure}
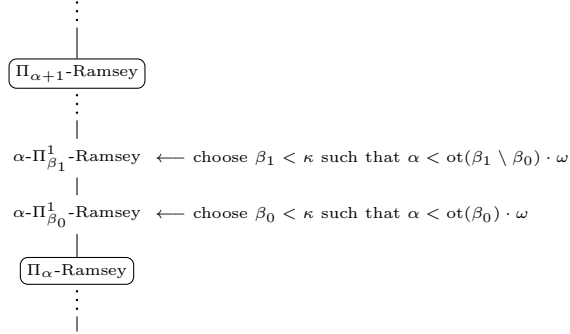

Next we use Theorem \ref{theorem_indescribability_in_infinite_ramseyness} to show that, if substantial care is taken, Theorem \ref{theorem_finite_ideal_diagram} can, in a sense, be extended to the ideals $\R^\alpha(\Pi^1_\beta(\kappa))$ for $\alpha>\omega$.

\begin{theorem}\label{theorem_infinite_ideal_diagram}
Suppose $\kappa$ is a cardinal, $\omega<\alpha<\kappa$ is a successor ordinal and $\beta<\kappa$ is an ordinal such that $\kappa\in\R^\alpha(\Pi^1_\beta(\kappa))^+$. Let $\delta$ be the greatest ordinal such that $\omega^\delta\leq\alpha$, let $m,n<\omega$ and $\gamma<\omega^\delta$ be the unique ordinals such that $\alpha=\omega^\delta m+\gamma+n+1$ where $\gamma$ is a limit ordinal.
\begin{enumerate}
\item If $m=1$ and $\gamma=0$ then
\[\R^\alpha(\Pi^1_\beta(\kappa))=\R^{\omega^\delta+n+1}(\Pi^1_\beta(\kappa))=\overline{\R_0(\R^{\omega^\delta+n}(\Pi^1_\beta(\kappa)))\cup\R^n(\Pi^1_{\beta+\omega^\delta+1}(\kappa))}.\]
\item Otherwise, if $m>1$ or $\gamma>0$, then
\begin{align*}\R^\alpha(\Pi^1_\beta(\kappa))&=\R^{\omega^\delta m+\gamma+n+1}(\Pi^1_\beta(\kappa))\\
	&=\overline{\R_0(\R^{\omega^\delta m+\gamma+n}(\Pi^1_\beta(\kappa)))\cup\R^{\omega^\delta (m-1)+\gamma+n+1}(\Pi^1_{\beta+\omega^\delta}(\kappa))}.
\end{align*}
\end{enumerate}
\end{theorem}

\begin{proof}
We proceed by induction on $\alpha$. The base case is $\alpha=\omega+1$. In this case $m=1$, $\gamma=0$ and $n=0$, so it suffices to show that 
\[\R^{\omega+1}(\Pi^1_\beta(\kappa))=\overline{\R_0(\R^\omega(\Pi^1_\beta(\kappa)))\cup\Pi^1_{\beta+\omega+1}(\kappa)},\]
but this follows directly from Theorem \ref{theorem_indescribability_in_infinite_ramseyness}.

We show that the result holds for $\alpha$ assuming it holds for all smaller successor ordinals. Suppose $\alpha=\omega^\delta m+\gamma+n+1$ as in the statement of the theorem. 

Let us show that (1) holds. Assume $m=1$ and $\gamma=0$. If $n=0$ then the result follows directly from Theorem \ref{theorem_indescribability_in_infinite_ramseyness}. Suppose $n\geq 1$. Let 
\[I=\overline{\R_0(\R^{\omega^\delta+n}(\Pi^1_\beta(\kappa)))\cup\R^n(\Pi^1_{\beta+\omega^\delta+1}(\kappa))}.\]
To prove that (1) holds we will show that $X\in\R^{\omega^\delta+n+1}(\Pi^1_\beta(\kappa))^+$ if and only if $X\in I^+$.

Suppose $X\in\R^{\omega^\delta+n+1}(\Pi^1_\beta(\kappa))^+$. By Remark \ref{remark_ideal_generated}, it will suffice to show that $X\in \R_0(\R^{\omega^\delta+n}(\Pi^1_\beta(\kappa)))^+$ and $X\in \R^n(\Pi^1_{\beta+\omega^\delta+1}(\kappa))^+$. By assumption, every regressive function $f:[X]^{<\omega}\to \kappa$ has a homogeneous set in $\R^{\omega^\delta+n}(\Pi^1_\beta(\kappa))^+$. By our inductive hypothesis we have
\[\R^{\omega^\delta+n}(\Pi^1_\beta(\kappa))=\overline{\R_0(\R^{\omega^\delta+n-1}(\Pi^1_\beta(\kappa)))\cup\R^{n-1}(\Pi^1_{\beta+\omega^\delta+1}(\kappa))}.\]
Thus every regressive function $f:[X]^{<\omega}\to \kappa$ has a homogeneous set in $\R^{n-1}(\Pi^1_{\beta+\omega^\delta+1}(\kappa))^+$, in other words, $X\in\R^n(\Pi^1_{\beta+\omega^\delta+1}(\kappa))^+$. Now let us show that $X\in\R_0(\R^{\omega^\delta+n}(\Pi^1_\beta(\kappa))^+$. Fix a regressive function $f:[X]^{<\omega}\to\kappa$ and a club $C\subseteq\kappa$. Since $X\in\R^{\omega^\delta+n+1}(\Pi^1_\beta(\kappa))^+$, there is a set $H\in\R^{\omega^\delta+n}(\Pi^1_\beta(\kappa))^+$ homogeneous for $f$. The fact that $H\in\R^{\omega^\delta+n}(\Pi^1_\beta(\kappa))^+$ is expressible over $(V_\kappa,\in,H)$ by a $\Pi^1_{\beta+\omega^\delta+2n}$ sentence $\varphi$. Since $X\cap C\in\R^{\omega^\delta+n+1}(\Pi^1_\beta(\kappa))^+$ and, by Theorem \ref{theorem_indescribability_in_infinite_ramseyness}, $\R^{\omega^\delta+n+1}(\Pi^1_\beta(\kappa))^+\subseteq\Pi^1_{\beta+\omega^\delta+2n+1}(\kappa)^+$, it follows that there is a $\xi\in X\cap C$ such that $H\cap \xi\in\R^{\omega^\delta+n}(\Pi^1_\beta(\xi))^+$. Hence $X\in\R_0(\R^{\omega^\delta+n}(\Pi^1_\beta(\kappa)))^+$.

Conversely, suppose $X\in I^+$. Let $f:[X]^{<\omega}\to \kappa$ be a regressive function. For the sake of contradiction, let us assume that every homogeneous set for $f$ is in $\R^{\omega^\delta+n}(\Pi^1_\beta(\kappa))$. By Lemma \ref{lemma_complexity}, this is expressible over $(V_\kappa,\in,X,f)$ by a $\Pi^1_{\beta+\omega^\delta+2n+1}$ sentence $\varphi$. Hence the set
\[C=\{\xi<\kappa\st (f\restrict\xi=f\cap V_\xi) \land (V_\xi,\in,X\cap V_\xi, f\cap V_\xi) \models\varphi\}\]
is in $\Pi^1_{\beta+\omega^\delta+2n+1}(\kappa)^*$. By Corollary \ref{corollary_indescribability_in_finite_ramseyness}, we have $\Pi^1_{\beta+\omega^\delta+2n+1}(\kappa)^*\subseteq\R^n(\Pi^1_{\beta+\omega^\delta+1}(\kappa))^*$, and so $C\in\R^n(\Pi^1_{\beta+\omega^\delta+1}(\kappa))^*$. Since $X\in I^+$ it follows that $X$ is not the union of a set in $\R_0(\R^{\omega^\delta+n}(\Pi^1_\beta(\kappa))$ and a set in $\R^n(\Pi^1_{\beta+\omega^\delta+1}(\kappa))$. Furthermore, since $X=(X\cap C)\cup(X\setminus C)$ and $X\setminus C\in\R^n(\Pi^1_{\beta+\omega^\delta+1}(\kappa))$, it follows that $X\cap C\in\R_0(\R^{\omega^\delta+n}(\Pi^1_\beta(\kappa))^+$. This implies that there is a $\xi\in X\cap C$ for which there is a set $H\subseteq X\cap C\cap\xi$ in $\R^{\omega^\delta+n}(\Pi^1_\beta(\xi))^+$ homogeneous for $f$. This contradicts the fact that $\xi\in C$. This establishes that (1) holds.

To show that (2) holds, suppose $m>1$ or $\gamma>0$. Let 
\[I=\overline{\R_0(\R^{\omega^\delta m+\gamma+n}(\Pi^1_\beta(\kappa)))\cup\R^{\omega^\delta (m-1)+\gamma+n+1}(\Pi^1_{\beta+\omega^\delta}(\kappa))}.\]
We will prove that $X\in\R^{\omega^\delta m+\gamma+n+1}(\Pi^1_\beta(\kappa))^+$ if and only if $X\in I^+$.

Suppose $X\in\R^{\omega^\delta m+\gamma+n+1}(\Pi^1_\beta(\kappa))^+$. This implies that every regressive function $f:[X]^{<\omega}\to \kappa$ has a homogeneous set in $\R^{\omega^\delta m+\gamma+n}(\Pi^1_\beta(\kappa))^+$. We will show that
\begin{align}
\R^{\omega^\delta m+\gamma+n}(\Pi^1_\beta(\kappa))^+\subseteq\R^{\omega^\delta (m-1)+\gamma+n}(\Pi^1_{\beta+\omega^\delta}(\kappa))^+.\tag{$*$}\label{equation_complicated_containment}
\end{align}
If $n\geq 1$ then by applying our inductive hypothesis to the successor ordinal $\alpha€™'=\omega^\delta m+\gamma+n<\alpha$, we obtain
\[\R^{\omega^\delta m+\gamma+n}(\Pi^1_\beta(\kappa))=\overline{\R_0(\R^{\omega^\delta m+\gamma+n-1}(\Pi^1_\beta(\kappa)))\cup \R^{\omega^\delta(m-1)+\gamma+n}(\Pi^1_\beta(\kappa))}\]
and thus (\ref{equation_complicated_containment}) holds. If $n=0$, to prove (\ref{equation_complicated_containment}) we must show that
\[\R^{\omega^\delta m+\gamma}(\Pi^1_\beta(\kappa))\supseteq\R^{\omega^\delta(m-1)+\gamma}(\Pi^1_{\beta+\omega^\delta}(\kappa)).\]
Choose $Z\in\R^{\omega^\delta (m-1)+\gamma}(\Pi^1_\beta(\kappa))$. Then there is a successor ordinal $\eta+k+1<\gamma$, where $\eta$ is a limit ordinal and $k<\omega$, such that $Z\in\R^{\omega^\delta(m-1)+\eta+k+1}(\Pi^1_{\beta+\omega^\delta}(\kappa))$. By our inductive hypothesis applied to the successor ordinal $\alpha'=\omega^\delta m+\eta+k+1<\alpha$, we have
\[\R^{\omega^\delta m+\eta+k+1}(\Pi^1_\beta(\kappa))=\overline{\R_0(\R^{\omega^\delta m+\eta+k}(\Pi^1_\beta(\kappa)))\cup\R^{\omega^\delta(m-1)+\eta+k+1}(\Pi^1_\beta(\kappa))}\]
and thus $Z\in\R^{\omega^\delta (m-1)+\eta+k+1}(\Pi^1_\beta(\kappa))\subseteq\R^{\omega^\delta m+\gamma}(\Pi^1_\beta(\kappa))$. This establishes (\ref{equation_complicated_containment}), which implies that every regressive function $f:[X]^{<\omega}\to \kappa$ has a homogeneous set in $\R^{\omega^\delta(m-1)+\gamma+n}(\Pi^1_{\beta+\omega^\delta}(\kappa))^+$, and hence $X\in\R^{\omega^\delta(m-1)+\gamma+n+1}(\Pi^1_\beta(\kappa))^+$.

Next, let us show that $X\in\R_0(\R^{\omega^\delta m+\gamma+n}(\Pi^1_\beta(\kappa)))^+$. Fix a regressive function $f:[X]^{<\omega}\to \kappa$ and a club $C\subseteq\kappa$. Since $X\in\R^{\omega^\delta m+\gamma+n+1}(\Pi^1_\beta(\kappa))^+$, there is a set $H\in\R^{\omega^\delta m+\gamma+n}(\Pi^1_\beta(\kappa))^+$ homogeneous for $f$. By Lemma \ref{lemma_complexity}, the fact that $H\in\R^{\omega^\delta m+\gamma+n}(\Pi^1_\beta(\kappa))^+$ can be expressed over $(V_\kappa,\in,X,f,H)$ by a $\Pi^1_{\beta+\omega^\delta m+\gamma+2n+1}$ sentence $\varphi$. Since $X\cap C\in\R^{\omega^\delta m+\gamma+n+1}(\Pi^1_\beta(\kappa))^+$ and, by Theorem \ref{theorem_indescribability_in_infinite_ramseyness}, 
\[\R^{\omega^\delta m+\gamma+n+1}(\Pi^1_\beta(\kappa))^+\subseteq\Pi^1_{\beta+\omega^\delta m +\gamma+2n+1}(\kappa)^+,\] it follows that there is a $\xi\in X\cap C$ such that $H\cap \xi\in\R^{\omega^\delta m +\gamma+n}(\Pi^1_\beta(\xi))^+$. This implies that $X\in\R_0(\R^{\omega^\delta m+\gamma+n}(\Pi^1_\beta(\kappa)))^+$. By Remark \ref{remark_ideal_generated}, this suffices to show that $X\in I^+$.

Conversely, suppose $X\in I^+$. Fix a regressive function $f:[X]^{<\omega}\to \kappa$. For the sake of contradiction, suppose every homogeneous set for $f$ is in $\R^{\omega^\delta m +\gamma+n}(\Pi^1_\beta(\kappa))$. This can be expressed over $(V_\kappa,\in,X,f)$ by a $\Pi^1_{\beta+\omega^\delta m+\gamma+2n+1}$ sentence $\varphi$. Hence, the set
\[C=\{\xi<\kappa\st (f\restrict\xi=f\cap V_\xi)\land(V_\xi,\in,X\cap V_\xi,f\cap V_\xi)\models\varphi\}\]
is in $\Pi^1_{\beta+\omega^\delta m+\gamma+2n+1}(\kappa)^*$. Since, by Theorem \ref{theorem_indescribability_in_infinite_ramseyness}, it follows that $\Pi^1_{\beta+\omega^\delta m +\gamma+2n+1}(\kappa)\subseteq\R^{\omega^\delta(m-1)+\gamma+n+1}(\Pi^1_{\beta+\omega^\delta}(\kappa))$, it follows that $C\in\R^{\omega^\delta(m-1)+\gamma+n+1}(\Pi^1_{\beta+\omega^\delta}(\kappa))^*$. Since $X=(X\cap C)\cup(X\setminus C)$ is not the union of a set in $\R_0(\R^{\omega^\delta m+\gamma+n}(\Pi^1_\beta(\kappa)))$ and a set in $\R^{\omega^\delta(m-1)+\gamma+n+1}(\Pi^1_{\beta+\omega^\delta}(\kappa))$, it follows that $X\cap C\in\R_0(\R^{\omega^\delta m+\gamma+n}(\Pi^1_\beta(\kappa)))^+$. This implies that there is a $\xi\in X\cap C$ for which there is a set $H\subseteq X\cap C\cap \xi$ in $\R^{\omega^\delta m +\gamma+n}(\Pi^1_\beta(\xi))^+$ homogeneous for $f$. This contradicts $\xi\in C$. This establishes (2).
\end{proof}

An argument similar to that of Theorem \ref{theorem_proper_containments_in_finite_diagram} can be used to show that the ideal containments suggested by the statement of Theorem \ref{theorem_infinite_ideal_diagram} are proper.

\begin{theorem}
Under the hypotheses of Theorem \ref{theorem_infinite_ideal_diagram}, the following hold.
\begin{enumerate}
\item If $m=1$ and $\gamma=0$ then 
\[\R^n(\Pi^1_{\beta+\omega^\delta+1}(\kappa))\subsetneq\R^{\omega^\delta+n+1}(\Pi^1_\beta(\kappa)).\]
\item If $m>0$ or $\gamma>0$ then 
\[\R^{\omega^\delta(m-1)+\gamma+n+1}(\Pi^1_{\beta+\omega^\delta}(\kappa))\subsetneq\R^{\omega^\delta m+\gamma+n+1}(\Pi^1_\beta(\kappa)).\]
\end{enumerate}
\end{theorem}

\begin{proof}
Since the containment follow easily from Theorem \ref{theorem_infinite_ideal_diagram}, it remains to show the properness of the containments.

For (1), let $S=\{\xi<\kappa\st\xi\in \R^n(\Pi^1_{\beta+\omega^\delta+1}(\kappa))\}$. By Lemma \ref{lemma_set_of_nons_is_positive}, $S\in\R^n(\Pi^1_{\beta+\omega^\delta+1}(\kappa))^+$. By Corollary \ref{corollary_ramseyness_reflects_indescribability}, it follows that $\kappa\setminus S\in \R^{n+1}([\kappa]^{<\kappa})^*\subseteq \R^{\omega^\delta+n+1}([\kappa]^{<\kappa})^*\subseteq \R^{\omega^\delta+n+1}(\Pi^1_\beta(\kappa))^*$. Thus $S\in\R^{\omega^\delta+n+1}(\Pi^1_\beta(\kappa))\setminus \R^n(\Pi^1_{\beta+\omega^\delta+1}(\kappa))$.

The argument for (2) is similar.
\end{proof}

\section{Generic embeddings}\label{section_generic_embeddings}

By considering properties of generic ultrapowers obtained by forcing with large cardinal ideals, we obtain characterizations of such ideals in terms of generic elementary embeddings. In what follows we obtain generic embedding characterizations of $\Pi^1_\beta$-indescribable as well as Ramsey subsets of cardinals. It should not be hard to find \emph{small embedding} \cite{MR3913154} characterizations of $\Pi^1_\beta$-indescribable sets which resemble our generic embeddings. However, it is not clear whether the characterization of Ramsey sets in Theorem \ref{theorem_generic_embedding_characterization_of_ramseyness} below can be rephrased in terms of small embeddings. Thus, the following question of \cite{MR3913154} remains open: can one characterize Ramsey cardinals using small embeddings?


Before providing a motivating example, let us recall a few basic facts about generic ultrapowers. If $\kappa$ is a regular uncountable cardinal and $I$ is an ideal on $\kappa$ and $S\in I^+$ then $I\restrict S=\{X\subseteq\kappa\st X\cap S\in I\}$ is an ideal on $\kappa$ extending $I$ and notice that $S\in (I\restrict S)^*$. We write $P(\kappa)/I$ to denote the usual atomless\footnote{We write $P(\kappa)/I$ when we really mean $P(\kappa)/I-\{[\emptyset]\}$.} boolean algebra obtained from $I$. If $G$ is $(V,P(\kappa)/I)$-generic then we let $U_G$ be the canonical $V$-ultrafilter obtained from $G$ extending the dual filter $I^*$. The appropriate version of {\L}os's Theorem can be easily verified, and thus we obtain a canonical generic elementary embedding $j:V\to V^\kappa/U_G$ in $V[G]$ where $j(x)=[\alpha\mapsto x]_U$. If $I$ is a normal ideal then the generic ultrafilter $U_G$ is $V$-normal and the critical point of the corresponding, possibly illfounded, generic ultrapower $j:V\to V^\kappa/U_G\subseteq V[G]$ is $\kappa$. When $I$ is a normal ideal, the corresponding generic ultrapower embedding $j$ is wellfounded on the ordinals up to $\kappa^+$. See \cite[Lemma 22.14]{MR1940513} or \cite[Section 2]{MR2768692} for more details.

\begin{definition}
When we say \emph{there is a generic elementary embedding $j:V\to M\subseteq V[G]$} we mean that there is some forcing poset $\P$ such that whenever $G$ is $(V,\P)$-generic then, in $V[G]$, there are definable classes $M$, $E$ and $j$ such that $j:(V,\in)\to (M,E)\subseteq V[G]$ is an elementary embedding, where $(M,E)$ is possibly not wellfounded.
\end{definition}

The following proposition is an easy application of generic ultrapowers obtained by forcing with $P(\kappa)/\NS_\kappa$. 

\begin{proposition}[Folklore]\label{proposition_stationarity_characterizations}
Suppose $\kappa>\omega$ is a regular cardinal. The following are equivalent.
\begin{enumerate}
\item $S\subseteq\kappa$ is stationary.
\item There is a generic elementary embedding $j:V\to M\subseteq V[G]$ with critical point $\kappa$ such that $\kappa \in j(S)$.
\end{enumerate}
\end{proposition}


It is natural to wonder: to what extent  can Proposition \ref{proposition_stationarity_characterizations} be generalized from the nonstationary ideal to other natural ideals, such as ideals associated to certain large cardinals?

\begin{proposition}\label{proposition_generic_embedding_indescribability}
Suppose $\kappa$ is a cardinal, $\beta<\kappa$ is an ordinal and $S\subseteq\kappa$. The following are equivalent.
\begin{enumerate}
\item $S$ is $\Pi^1_\beta$-indescribable.
\item There is a generic elementary embedding $j:V\to M\subseteq V[G]$ with critical point $\kappa$ such that $\kappa\in j(S)$ and for all $A\in V_{\kappa+1}^V$ and all $\Pi^1_\beta$ sentences $\varphi$ we have
\[((V_\kappa,\in,A)\models\varphi)^V\implies ((V_\kappa,\in,A)\models\varphi)^M.\]
\end{enumerate}
\end{proposition}

\begin{proof}
Suppose $S$ is $\Pi^1_\beta$-indescribable. Let $G\subseteq P(\kappa)/(\Pi^1_\beta(\kappa)\restrict S)$ be generic over $V$ and let $j:V\to M:=V^\kappa/G$ be the corresponding generic ultrapower embedding. Since $S\in G$ we have $\kappa\in j(S)$. Fix $A\in V_{\kappa+1}^V$ and fix a $\Pi^1_\beta$ sentence $\varphi$ such that $((V_\kappa,\in,A)\models\varphi)^V$. Since the set
\[C:=\{\xi<\kappa\st(V_\xi,\in,A\cap V_\xi)\models\varphi\}\]
is in the filter $\Pi^1_\beta(\kappa)^*$, it follows that $S\cap C\in (\Pi^1_\beta(\kappa)\restrict S)^*\subseteq G$, and thus $\kappa\in j(C)$. This implies $((V_\kappa,\in,A)\models\varphi)^M$.

Conversely, suppose $j:V\to M$ is a generic elementary embedding satisfying $(2)$. Let us show that $S$ is $\Pi^1_\beta$-indescribable. Fix an $A\in V_{\kappa+1}^V$ and a $\Pi^1_\beta$ sentence $\varphi$ such that $((V_\kappa,\in,A)\models\varphi)^V$. By elementarity and by (2), there is some $\xi\in S$ such that $((V_\xi,\in,A\cap V_\xi)\models\varphi)^V$, thus $S$ is $\Pi^1_\beta$-indescribable.
\end{proof}

Let us show that the ideals $\R^m(\Pi^1_\beta(\kappa))$ can be characterized in terms of generic elementary embeddings. Taking $m=1$ and $\beta=-1$ in the following theorem yields a characterization of the Ramsey ideal and of Ramsey cardinals.

\begin{theorem}\label{theorem_generic_embedding_characterization_of_ramseyness}
Suppose $\kappa$ is a cardinal, $1\leq m<\omega$, $\beta\in\{-1\}\cup\kappa$ and $S\subseteq\kappa$. The following are equivalent.
\begin{enumerate}
\item $S\in\R^m(\Pi^1_\beta(\kappa))^+$
\item There is a generic elementary embedding $j:V\to M$ with critical point $\kappa$ such that $\kappa\in j(S)$ and the following properties hold.
\begin{enumerate}
\item For all $A\in V_{\kappa+1}^V$ and all $\Pi^1_{\beta+2m}$ sentences $\varphi$ we have
\[((V_\kappa,\in,A)\models\varphi)^V\implies ((V_\kappa,\in,A)\models\varphi)^M.\]
\item For all regressive functions $f:[S]^{<\omega}\to \kappa$ in $V$ we have
\[M\models(\exists H\in \R^{m-1}(\Pi^1_\beta(\kappa))^+)(\text{$H$ is homogeneous for $f$}).\]
\end{enumerate}
\end{enumerate}
\end{theorem}

\begin{proof}
Suppose $S\in\R^m(\Pi^1_\beta(\kappa))^+$. Let $G\subseteq P(\kappa)/(\R^m(\Pi^1_\beta(\kappa))\restrict S)$ be generic over $V$ and let $j:V\to M:=V^\kappa/G$ be the corresponding generic ultrapower. Since $S\in G$ we have $\kappa\in j(S)$. By Corollary \ref{corollary_indescribability_in_finite_ramseyness}, we have
\[\R^m(\Pi^1_\beta(\kappa))=\overline{\R_0(\R^{m-1}(\Pi^1_\beta(\kappa)))\cup\Pi^1_{\beta+2m}(\kappa)}.\]
Since $\Pi^1_{\beta+2m}(\kappa)^*\subseteq\R^m(\Pi^1_\beta(\kappa))^*\subseteq G$ it follows, by an argument similar to that in the proof of Proposition \ref{proposition_generic_embedding_indescribability}, that (a) holds. Fix a regressive function $f:[S]^{<\omega}\to \kappa$ in $V$. Since $S\in\R^m(\Pi^1_\beta(\kappa))^+$ there is a set $H\in P(S)\cap\R^{m-1}(\Pi^1_\beta(\kappa))^+\cap V$ which is homogeneous for $f$. Clearly $H=j(H)\cap \kappa$ and $f=j(f)\cap (\kappa\times\kappa)$ are in $M$, and $M$ thinks that $H$ is homogeneous for $f$. By Lemma \ref{lemma_complexity} the fact that $H\in\R^{m-1}(\Pi^1_\beta(\kappa))^+$ is expressible by a $\Pi^1_{\beta+2m}$ sentence over $V_\kappa$, and thus by (a) we see that $M\models$ ``$H\in\R^{m-1}(\Pi^1_\beta(\kappa))^+$''.

Conversely, suppose (2) holds. Fix a regressive function $f:[S]^{<\omega}\to \kappa$ in $V$. For the sake of contradiction suppose that in $V$, every subset of $S$ which is homogeneous for $f$ is in the ideal $\R^{m-1}(\Pi^1_\beta(\kappa))$. By Lemma \ref{lemma_complexity}, this can be expressed by a $\Pi^1_{\beta+2m}$ sentence over $V_\kappa$, thus by (2)(a), $M$ thinks that every homogeneous set for $f$ is in the ideal $\R^{m-1}(\Pi^1_\beta(\kappa))$. This contradicts (2)(b).
\end{proof}

Using Lemma \ref{lemma_complexity} and Theorem \ref{theorem_indescribability_in_infinite_ramseyness}, an argument similar to that of Theorem \ref{theorem_generic_embedding_characterization_of_ramseyness} gives a generic embedding characterization of certain ideals of the form $\R^\alpha(\Pi^1_\beta(\kappa))$ for $\alpha>\omega$.


\end{document}